\documentclass[11pt]{article}
\usepackage{amsmath,amssymb,amsthm,amscd}
\usepackage[mathscr]{eucal}
\newtheorem{theorem}{Theorem}[section]
\newtheorem{lemma}[theorem]{Lemma}
\newtheorem{proposition}[theorem]{Proposition}

\theoremstyle{plain}

\theoremstyle{definition}
\newtheorem{definition}[theorem]{Definition}

\numberwithin{equation}{section}

\newcommand{\Int}{\operatorname{Int}}

\newcommand{\Ad}{\operatorname{Ad}}

\newcommand{\G}{\mathcal{G}}

\newcommand{\N}{\mathbb{N}}
\newcommand{\T}{\mathbb{T}}
\newcommand{\Z}{\mathbb{Z}}
\newcommand{\Zp}{{\mathbb{Z}}_+}

\def\L{\mathcal{L}}
\def\G{\mathcal{G}}

\def\F{\mathcal{F}}

\def\X{\mathcal{X}}

\def\CZ{{\mathcal{Z}}}

\def\R{\mathcal{R}}

\def\Max{{{\operatorname{Max}}}}

\def\Ext{{{\operatorname{Ext}}}}
\def\Ad{{{\operatorname{Ad}}}}

\def\det{{{\operatorname{det}}}}
\def\exp{{{\operatorname{exp}}}}
\def\red{{{\operatorname{red}}}}
\def\full{{{\operatorname{full}}}}

\def\LGBS{({\mathcal{L}}^-, {\mathcal{L}}^+)}
\def\L{{\mathcal{L}}}

\def\XL{\mathcal{X}_{\mathcal{L}}}
\def\ZL{\mathcal{Z}_{\mathcal{L}}}
\def\GL{\mathcal{G}_{\mathcal{L}}}
\def\RL{\mathcal{R}_{\mathcal{L}}}

\def\FL{{{\mathcal{F}}_{{\mathcal{L}}}}}

\def\mbbZ2{\mathbb{Z}^2[<]}

\title{An \'etale equivalence relation on a configuration space arising from a subshift
and related $C^*$-algebras
\\
}
\author{Kengo Matsumoto \\
Department of Mathematics \\ Joetsu University of Education \\
Joetsu, 943-8512, Japan }

\begin{document}
\maketitle

\date{}

\def\det{{{\operatorname{det}}}}

\begin{abstract}   
A $\lambda$-graph bisystem $\mathcal{L}$ 
 consists of two labeled Bratteli diagrams $(\mathcal{L}^-,\mathcal{L}^+)$, 
that presents a two-sided subshift $\Lambda_\mathcal{L}$.
We will construct a compact totally disconnected metric space  
with a shift homeomorphism consisting  of two-dimensional configurations 
 from a $\lambda$-graph bisystem.
The configuration space has 
a certain AF-equivalence relation written $R_\L$ 
with a natural shift homeomorphism
$\sigma_\L$ coming from the  shift homeomorphism $\sigma_{\Lambda_{\mathcal{L}}}$ 
on the subshift $\Lambda_\L$. 
The equivalence relation $R_\L$ yields an AF-algebra $\FL$
with an automorphism $\rho_\L$ on it. 
We will study invariance of the \'etale equivalence relation $R_\L$,
the groupoid $\G_\L=R_\L\rtimes_{\sigma_\L}\Z$
and the groupoid $C^*$-algebras 
$C^*(R_\L)$,  $C^*(\G_\L)$
under topological conjugacy of the presenting two-sided subshifts. 
 \end{abstract}



\section{Introduction}
In \cite{MaPre2019a}, the author introduced a notion of $\lambda$-graph bisystem $\L$
that consists of two kinds of labeled Bratteli diagrams 
$\L^-, \L^+$ satisfying a certain compatibility condition between them.  
It is a graphical presentation of a two-sided subshift.
In a recent paper \cite{MaPre2019b},
he constructed an AF $C^*$-algebra with shift automorphism
from $\lambda$-graph bisystem, 
that is regarded as a subshift version of the asymptotic Ruelle algebra for shifts of finite type.
Asymptotic Ruelle algebras have been introduced  in Ruelle's paper \cite{Ruelle1}
and initiated to study by I. Putnam \cite{Putnam1}.
The Ruelle algebras are constructed from hyperbolic dynamical systems called Smale spaces.   
We have to emphasize that a subshift can not be any Smale space 
unless it is a shift of finite type.
Hence the construction of the Ruelle algebras 
by using asymptotic \'etale groupoids from shifts of finite type in 
\cite{Ruelle1} and \cite{Putnam1}, \cite{Putnam2} (cf. \cite{PutSp})
does not work for general subshifts.
That is a strong motivation of producing the presentation of the paper.

In this paper, we will construct a configuration space $\XL$ consists of two-dimensional tilings
 from a $\lambda$-graph bisystem $\L =(\L^-,\L^+)$ presenting a subshift.
The configuration space is a compact totally disconnected metric space
having a certain AF-equivalence relation written $R_\L$ on $\XL$ 
that yields the AF-algebra $\FL$
introduced in the previous paper \cite{MaPre2019b}.
The configuration space also has a natural shift homeomorphism
$\sigma_\L$ coming from the  shift homeomorphism $\sigma_{\Lambda_\L}$ 
on the subshift $\Lambda_\L$ presented by the $\lambda$-graph bisystem $\L$. 
There is a factor map  
$\pi: (\XL,\sigma_\L)\longrightarrow (\Lambda_\L,\sigma_{\Lambda_\L})$
that connects the dynamics of the subshift $\Lambda_\L$ to that of the
configuration space $\XL$. 
We will construct an \'etale groupoid $\GL$ from the equivalence relation
$R_\L$ with the shift homeomorphism $\sigma_\L$ by
\begin{equation*}
\mathcal{G}_{\L} =\{
(x,n,z) \in \mathcal{X}_{\mathcal{L}}\times\Z\times\mathcal{X}_{\mathcal{L}}
\mid (\sigma_{\mathcal{L}}^n(x), z) \in R_{\mathcal{L}} \}. 
\end{equation*}
Its unit space $\GL^{(0)} =\{(x,0,x)\in \GL\}$ is identified with $\XL$.  
The groupoid $\GL$ is regarded as a generalization of  groupoids constructed from asymptotic equivalence relation on two-sided shifts of finite type introduced in \cite{Ruelle1} and \cite{Putnam1}, \cite{Putnam2} (cf. \cite{PutSp}) to general two-sided subshifts .

Let us denote by $\R_\L$ 
the groupoid $C^*$-algebra $C^*(\GL)$.
The shift homeomorphism 
$\sigma_\L$ on $\XL$ also gives rise to an automorphism written 
$\rho_\L$ on the AF-algebra $\FL$.
We then know that  the $C^*$-algebra $\R_\L$ is isomorphic to the crossed product
$\FL\rtimes_{\rho_\L}\Z$.
Denote by $\gamma_\L$ the dual action $\hat{\rho}_\L$
of the crossed product $\FL\rtimes_{\rho_\L}\Z$.
Under certain  essential freeness condition called equivalently essential freeness of the topological dynamical system
$(\X_\L, \sigma_\L)$, we may prove the following theorem.
\begin{theorem}[Theorem \ref{thm:GRmain}]
Suppose that $(\X_{\mathcal{L}_i},\sigma_{\mathcal{L}_i}), i=1,2$ 
are equivalently essentially free.
The following two assertions are equivalent.
\hspace{7cm}
\begin{enumerate}
\renewcommand{\theenumi}{\roman{enumi}}
\renewcommand{\labelenumi}{\textup{(\theenumi)}}
\item
There exists an isomorphism 
$\varphi:\mathcal{G}_{\mathcal{L}_1}\longrightarrow  \mathcal{G}_{\mathcal{L}_2}$
of the \'etale groupoids such that
$d_{\mathcal{L}_2} =\varphi\circ d_{\mathcal{L}_1}$,
where $d_{\mathcal{L}_i}:\mathcal{G}_{\mathcal{L}_i}\longrightarrow\Z$
is defined by $d_{\mathcal{L}_i}(x,n,z) = n$ for $(x,n,z) \in \mathcal{G}_{\mathcal{L}_i}$.
\item
There exists an isomorphism 
$\Phi: \R_{\mathcal{L}_1} \longrightarrow \R_{\mathcal{L}_2}
$ of $C^*$-algebras such that
$\Phi(C({\mathcal{X}}_{\mathcal{L}_1})) =C({\mathcal{X}}_{\mathcal{L}_2})$
and
$\Phi\circ\gamma_{{\L_1}_t}=\gamma_{{\L_2}_t}\circ\Phi$,
$t \in \T$.
\end{enumerate}
\end{theorem}
Since any two-sided subshift $\Lambda$ 
associates to a $\lambda$-graph bisystem in a canonical way,
that is written $\L_\Lambda = (\L_\Lambda^-,\L_\Lambda^+)$,
and called the canonical $\lambda$-graph bisystem for subshift $\Lambda$.
Its presenting subshift $\Lambda_{\L_\Lambda}$ coincides with 
the original subshift $\Lambda$.
Hence any subshift yields the configuration space 
$\X_{\L_\Lambda}$ with shift homeomorphism $\sigma_{\L_\Lambda}$, 
the equivalence relation $R_{\L_\Lambda}$, 
the AF-algebra $\F_{\L_\Lambda}$ with shift automorphism $\rho_{\L_\Lambda}$
and the $C^*$-algebra $\R_{\L_\Lambda}$.  
They are written 
$\X_\Lambda, \sigma_{\X_\Lambda}, R_\Lambda, \F_\Lambda, \rho_\Lambda
$ and $\R_\Lambda$, respectively.
Through the factor map $\pi:\X_\Lambda\longrightarrow \Lambda$,
the commutative $C^*$-algebra $C(\Lambda)$ of complex valued continuous functions
on $\Lambda$ is regarded as a $C^*$-subalgebra of $C(\X_\Lambda)$ that locates 
in the diagonal of the AF-algebra $\F_{\Lambda}$.
We then have inclusions
$$
C(\Lambda) \subset C(\X_\Lambda) \subset \F_\Lambda \subset \R_\Lambda
$$  
of $C^*$-subalgebras.
The constructions of the configuration space with shift homeomorphism,
the \'etale equivalence relation are dynamical (Proposition \ref{prop:topconjugate1}). 
Hence the isomorphism class of the \'etale groupoid is also a dynamical object
(Proposition \ref{prop:topconjugate2}).
Under certain essential freeness condition mentioned above on the dynamical system
$(\X_\Lambda,\sigma_{\X_\Lambda})$, we will show the following result.
\begin{theorem}[Theorem \ref{thm:GRmain2} and Theorem \ref{thm:GRmain3}] \label{thm:main1}
Let  $\Lambda_1, \Lambda_2$  be two-sided subshifts.
Suppose that $(\X_{\Lambda_i},\sigma_{\X_{\Lambda_i}}), i=1,2$ are equivalently essentially free.
The two-sided subshifts $\Lambda_1, \Lambda_2$ are topologically conjugate
if and only if there exists an isomorphism  
$\Phi:  \R_{\Lambda_1} \longrightarrow \R_{\Lambda_2}$ 
of $C^*$-algebras such that 
$$
\Phi(C(\X_{\Lambda_1}) )= C(\X_{\Lambda_2}), \qquad
\Phi(C(\Lambda_1) = C(\Lambda_2), \qquad
\Phi\circ \gamma_{{\Lambda_1}_t} 
= \gamma_{{\Lambda_2}_t}\circ\Phi, \quad t \in \mathbb{T}.
$$
Hence the quadruplet  
$(\R_\Lambda, C(\X_\Lambda), C(\Lambda), \gamma_\Lambda)$ 
is a complete invariant of topological conjugacy of subshift $\Lambda$
under the condition that $(\X_\Lambda,\sigma_{\X_\Lambda})$ is equivalently essentially free.
If in particular, a subshift $\Lambda$ satisfy $\pi$-condition (I), then 
the triplet  
$(\R_\Lambda,  C(\Lambda), \gamma_\Lambda)$ is a complete invariant of 
topological conjugacy of subshift $\Lambda$.
\end{theorem}
If a subshift is a shift of finite type, our construction of the groupoid 
reduces to the  \'etale groupoids of asymptotic equivalence relations on the shift of finite type
introduced in Ruelle \cite{Ruelle1} and Putnam \cite{Putnam1}, 
and our construction coincides with the Ruelle-Putnam's construction 
of the asymptotic groupoids from shifts of finite type.     
If we restrict our interest to irreducible shifts of finite type, 
we have the following theorem as a special case of the above theorem.
\begin{theorem}[{Theorem \ref{thm:mainmarkov}, \cite[Theoren 3.3]{Ma2019JOT}}]
Let $A, B$ be irreducible, non-permutation  matrices with entries in $\{0,1\}$. 
Then the following two conditions are equivalent:
\begin{enumerate}
\renewcommand{\theenumi}{\roman{enumi}}
\renewcommand{\labelenumi}{\textup{(\theenumi)}}
\item
The two-sided topological Markov shifts  
$(\Lambda_A, \sigma_A)$ and $(\Lambda_B, \sigma_B)$ are topologically conjugate.
\item
There exists an isomorphism 
$\Phi: \R_{\Lambda_A} \longrightarrow \R_{\Lambda_B}$
of $C^*$-algebras such that
$\Phi(C({\Lambda_A})) =C({\Lambda_B})$
and $\Phi\circ\gamma_{A_t} =\gamma_{B_t} \circ\Phi$,
$t \in \T$.
\end{enumerate}
\end{theorem}
The above result for topological Markov shifts were already seen in \cite[Theoren 3.3]{Ma2019JOT}.

As in Theorem \ref{thm:main1}, the K-groups 
for the $C^*$-algebras $\F_\L, \R_\L$ 
give rise to topologically conjugacy invariants for two-sided subshifts.
In particular, the $K_0$-group $K_0(\FL)$ is regarded as a bilateral dimension group
for the presenting subshift $\Lambda_\L$.
Their  K-group formulas  were  studied in \cite{MaPre2019b}.

Throughout the paper, $\mathbb{N}, \Zp$ denote the set of positive integers, and 
that of nonnegative integers, respectively.

\section{Subshifts and $\lambda$-graph bisystems}
Let $\Sigma$ be a finite set, called an alphabet.
 A labeled Bratteli diagram $(V, E, \lambda)$
 in general consists of vertex set $V= \cup_{l=0}^\infty V_l$,
 edg set $E=\cup_{l=0}^\infty E_{l,l+1}$ and labeling map
 $\lambda:E\longrightarrow \Sigma$,
where $V_l, l\in \Zp $ as well as $E_{l,l+1}, l\in \Zp$
are finite, nonempty pairwise disjoint sets.
We always assume that the top vertex set $V_0$ is a singleton.
Let $\Sigma^-, \Sigma^+$ be two finite sets. 
A $\lambda$-graph bisystem
$\mathcal{L}$ consists of two labeled Bratteli diagrams
$\mathcal{L}^- = (V, E^-, \lambda^-)$
and
$\mathcal{L}^+ = (V, E^+, \lambda^+)$ 
with common vertex set $V$
satisfying the following properties.
The first labeled Bratteli diagram 
$\mathcal{L}^- = (V, E^-, \lambda^-)$
is equipped with the maps 
$s :E^-_{l+1,l} \longrightarrow V_{l+1}$
and 
$t :E^-_{l+1,l} \longrightarrow V_{l}$
for which we call the source $s(e^-) \in V_{l+1}$
and
the terminal $t(e^-) \in V_l$ for an edge $e^- \in E_{l+1,l}^-,$
and a lebeling map
$\lambda^-:E_{l+1,l}^-\longrightarrow \Sigma^-$.
The other labeled Bratteli diagram 
$\mathcal{L}^+ = (V, E^+, \lambda^+)$
is equipped with the maps $s :E^+_{l,l+1} \longrightarrow V_{l}$
and $t :E^+_{l,l+1} \longrightarrow V_{l+1}$
for which we call the source $s(e^+) \in V_{l}$
and
the terminal $t(e^+) \in V_{l+1}$ for an edge $e^+ \in E_{l,l+1}^+,$
and a lebeling map
$\lambda^+:E_{l,l+1}^+\longrightarrow \Sigma^+$.
Hence the two labeled Bratteli diagrams have a common vertex sets
$V = \cup_{l=0}^\infty V_l$,
and different edges with different directions and edge labeling.
Their edges labeling are required to satisfy the following property called the local property
of $\lambda$-graph bisystem. 
For any two vertices $u \in V_l, v \in V_{l+2}$,
we set 
\begin{align*}
E^{+}_-(u,v) & =\{
(f^+, f^-) \in E^+_{l,l+1}\times E^-_{l+2,l+1} \mid
s(f^+) = u,\, t(f^+) = t(f^-),\, s(f^-) = v\}, \\ 
E^{-}_+(u,v) & =\{
(e^-, e^+) \in E^-_{l+1,l}\times E^+_{l+1,l+2} \mid
t(e^-) = u,\, s(e^-) = s(e^+),\,t(e^+) = v\}.
\end{align*}
The local property requires us that there exists
a bijective correspondence 
$\phi:E^{+}_-(u,v)\longrightarrow E^{-}_+(u,v)$
such that $\lambda^+(f^+) =\lambda^+(e^+),\,
\lambda^-(f^-) =\lambda^-(e^-)$, whenever  $\phi(f^+,f^-) = (e^+,e^-)$.
We further assume that 
the labeled Bratteli diagram $\mathcal{L}^-$ is right-resolving, that is,
if two distinct edges $e^-, f^- \in E_{l+1,l}^-$ satisfy $s(e^-) = s(f^-)$, 
then $\lambda^-(e^-) \ne \lambda^-(f^-)$, 
as well as 
$\mathcal{L}^+$ is left-resolving, that is,
if two distinct edges $e^+, f^+ \in E_{l,l+1}^+$ satisfy $t(e^+) = t(f^+)$, 
then $\lambda^+(e^+) \ne \lambda^+(f^+)$. 
We then call the pair $\LGBS$ a $\lambda$-graph bisystem.
It is sometimes written as $\L =\LGBS$.
In this paper, we further assume the following extra condition called FPCC for a $\lambda$-graph bisystem.
Let us assume that the alphabet sets 
$\Sigma^-, \Sigma^+$
are the same $\Sigma^- = \Sigma^+$ written $\Sigma$.
Let $\{v_1^l,\dots,v_{m(l)}^l\}$ be the vertex set $V_l$.
For $v_i^l\in V_l$, 
we define its follower set $F(v_i^l)$ in $\mathcal{L}^-$
and its predecessor set $P(v_i^l)$ in $\mathcal{L}^+$ by setting
\begin{align*}
F(v_i^l)  =& \{ (\lambda^-(e^-_l), \lambda^-(e^-_{l-1}),\cdots,\lambda^-(e^-_1) ) \in \Sigma^l \mid \\
  & \qquad e^-_n \in E^-_{n,n-1}, s(e^-_l) = v_i^l, t(e^-_n) = s(e^-_{n-1}), n=2,3,\dots,l \}, \\ 
P(v_i^l) = & \{ (\lambda^+(e^+_1), \lambda^+(e^+_{2}),\cdots,\lambda^+(e^+_l)) \in \Sigma^l  \mid \\
  & \qquad e^+_n \in E^+_{n-1,n}, t(e^+_l) = v_i^l, s(e^+_n) = t(e^+_{n-1}), n=2,3,\dots,l \}.
\end{align*}
The $\lambda$-graph bisystem is said to satisfy FPCC (Follower-Predecessor Compatibility Condition)
if  $F(v_i^l) =P(v_i^l)$ for all $v_i^l \in V_l.$
Throughout the paper, we always assume that 
$\lambda$-graph bisystems satisfy FPCC.
 
For a $\lambda$-graph bisystem $\LGBS$ over alphabet $\Sigma$,
let us denote by $\Lambda_\L$ the subshift 
whose admissible words are the set of words
$\cup_{v_i^l \in V}F(v_i^l)  (=\cup_{v_i^l \in V}P(v_i^l))$.
We say that the $\lambda$-graph bisystem  $ \L$
presents the subshift  $\Lambda_{\mathcal{L}}$. 
Conversely any subshift is presented by a $\lambda$-graph bisystem satisfying FPCC.
We briefly explain how to construct a $\lambda$-graph bisystem from an arbitrary subshift $\Lambda$
following \cite{MaPre2019a} and \cite{MaPre2019b}.
Let us denote by $\Z^2[<]$ the set of 
$(k,l)\in \Z^2$ such that $k<l$. 
For $n \in \Zp$, 
we denote by
$B_n(\Lambda)$ the set of admissible words of $\Lambda$ with lengh $n$.
For $x \in \Lambda$ and
$(k,l)\in \Z^2[<]$, 
the set $W_{k,l}(x)$ of $(k,l)$-central words of $x$ 
is defined by
\begin{align*}
W_{k,l}(x) = & \{(a_{k+1},a_{k+2}, \cdots, a_{l-1}) \in B_{n(k,l)}(\Lambda) \mid \\
& \qquad
(\cdots, x_{k-1}, x_k, a_{k+1},a_{k+2}, \cdots, a_{l-1}, x_l,x_{l+1}, \cdots) \}, 
\end{align*}
where
$n(k,l) = l - k-1$.
For $x, z \in \Lambda$, we call 
$x, z$ are $(k,l)$-centrally equivalent,
written $x\overset{c}{\underset{(k,l)}{\sim}}z$
if $
W_{k,l}(x) = W_{k,l}(z).
$ 
The set 
$
\Lambda/\overset{c}{\underset{(k,l)}{\sim}}
$ 
of $(k,l)$-centrally equivalence classes 
of $\Lambda$ is denoted by 
$\Omega_{k,l}$.
Let $m(k,l)$ be the cardinal number of the set 
$\Omega_{k,l}$, that is finite.
As
$$
x\overset{c}{\underset{(k,l)}{\sim}}z\qquad
 \Longleftrightarrow \qquad
\sigma(x)\overset{c}{\underset{(k-1,l-1)}{\sim}}\sigma(z),
$$
the equivalence classes 
$\Omega_{k,l}$ and $\Omega_{k-n,l-n}$
are naturally identified for any $n\in\Z$.
We define the vertex set
$V_n, n\in \Zp$ by 
$V_{n(k,l)} = \Omega_{k,l}$, for example
\begin{gather*}
V_0:= \{ \Lambda\},\quad 
 V_1:= \Omega_{-1,1},\quad
 V_2:= \Omega_{-2,1}(= \Omega_{-1,2}), \quad
\cdots,  \quad  
V_l:= \Omega_{-l,1}(= \Omega_{-1,l}), \quad \cdots.
\end{gather*}
The sequence $V_l, l\in \Zp$ 
will be the vertex sets of a $\lambda$-graph bisystem.
Let us denote by $m(l)$ the cardinality of the vertex set $V_l$,
so that $m(n(k,l)) = m(k,l)$ for $(k,l) \in \Z^2[<]$. 
Since $V_l$ consists of $m(l)$ equivalence classes
that are denoted by 
$\{ v_1^l, \dots, v_{m(l)}^l\}$.
For 
$x=(x_n)_{n\in \Z} \in v_i^{n(k,l)}$ and $\beta \in \Sigma$
such that 
$ (\beta, a_{k+2}, a_{k+3},\dots, a_{l-1}) \in W_{k,l}(x)$
for some 
$( a_{k+2},a_{k+3},\dots,a_{l-1}) \in B_{l-k-2}(\Lambda)$,
if 
\begin{equation*}
 (\dots, x_{k-1}, x_k, \beta,a_{k+2}, a_{k+3},\dots, a_{l-1}, x_l, x_{l+1},\dots ) \in \Lambda
\end{equation*}
belongs to the equivalence class
$v_h^{n(k+1,l)}\in V_{n(k+1,l)} (= V_{n(k,l)-1})$,
then
we define a directed edge $e^- :
v_i^{n(k,l)} \overset{e^-}{\underset{\beta}{\longrightarrow}} v_h^{n(k,l)-1}
$ 
labeled $\beta$,
that is 
$s(e^-) = v_i^{n(k,l)}, 
t(e^-) =v_h^{n(k,l)-1}
$ and 
$\lambda^-(e^-) = \beta$.

Similarly, for
$x=(x_n)_{n\in \Z} \in v_i^{n(k,l)}$ and $\alpha \in \Sigma$
such that  
$ (a_{k+1},a_{k+2},\dots,a_{l-2},\alpha) \in W_{k,l}(x)$ 
for some $( a_{k+1}, a_{k+2},\dots, a_{l-2}) \in B_{l-k-2}(\Lambda)$,
if 
\begin{equation*}
(\dots, x_{k-1}, x_k, a_{k+1}, a_{k+2},\dots, a_{l-2},\alpha, x_l, x_{l+1},\dots ) \in \Lambda
\end{equation*}
belongs to the equivalence class
$v_j^{n(k,l-1)} \in V_{n(k,l-1)}(=V_{n(k,l)-1}),$
we define a directed edge
$e^+: v_j^{n(k,l)-1}\overset{e^+}{\underset{\alpha}{\longrightarrow}}v_i^{n(k,l)}
$
labeled $\alpha$,
that is 
$s(e^+)=v_j^{n(k,l)-1},
t(e^+) = v_i^{n(k,l)}
$
and
$
\lambda^+(e^+) = \alpha$.

Let us denote by
$E_{l+1,l}^-$
the set of directed edges
from a vertex $v_j^{l+1}\in V_{l+1}$ to a vertex $v_i^{l}\in V_l$
for some $i=1,\dots,m(l), \, j=1,\dots, m(l+1)$.
Similarly
the set of directed edges 
from a vertex
$v_i^l\in V_l$ to a vertex $v_j^{l+1}\in V_{l+1}$
for some $i=1,\dots,m(l), \, j=1,\dots, m(l+1)$
is denoted by
$E_{l,l+1}^+$.
We then put
$$
E^- =\cup_{l=0}^{\infty} E_{l+1,l}^-, \qquad
E^+ =\cup_{l=0}^{\infty} E_{l,l+1}^+
$$
so that the pair $({\mathcal{L}}^-_\Lambda, {\mathcal{L}}^+_\Lambda)$ 
of the resulting labeled Bratteli diagrams 
${\mathcal{L}}^-_\Lambda =(V,E^-, \lambda^-)$ and
${\mathcal{L}}^+_\Lambda =(V,E^+, \lambda^+)$
become a $\lambda$-graph bisystem satisfying
FPCC (\cite{MaPre2019a}, \cite{MaPre2019b}).
The $\lambda$-graph bisystem is written 
${\mathcal{L}}_\Lambda=({\mathcal{L}}^-_\Lambda, {\mathcal{L}}^+_\Lambda)$
and called the canonical $\lambda$-graph bisystem for
$\Lambda$.
The presenting subshift $\Lambda_{{\mathcal{L}}_\Lambda}$ of ${\mathcal{L}}_\Lambda$
coincides with the original subshift $\Lambda$.

\section{Configuration space}

Let $\L =\LGBS$ be a $\lambda$-graph bisystem over $\Sigma$ satisfying FPCC
such that $\L^- =(V,E^-,\lambda^-)$ and
$\L^+ =(V,E^+,\lambda^+)$.
We will introduce a set $\XL=(V_\L, E_\L, \lambda_\L)$ of 
two-dimensional configurations
 expanded in  lower right infinitly in the two-dimensional plane. 
Recall that  
$\mbbZ2$ denotes 
the set of $(k,l) \in \mathbb{Z}^2$ such that $k<l$.
For $(k,l) \in \mbbZ2$, we put  $n(k,l) = l-k-1$.
The vertex set $V_{(k,l)}$ at $(k,l) \in \mathbb{Z}^2$
is defined by the vertex set 
$V_{n(k,l)}$ at level $n(k,l)$ of $\L$,
and we set
\begin{equation*}
V_\L = \bigcup_{(k,l)\in \mbbZ2} V_{(k,l)}.  
\end{equation*}
We define two kinds of edge sets 
$E^-_{(k-1,k),l}, E^+_{k,(l,l+1)} $ for $(k,l) \in \mbbZ2$ by using the edge sets of $\L$ 
as follows:
\begin{equation*}
E^-_{(k-1,k),l} =E^-_{n(k-1,l), n(k,l)},\qquad 
E^+_{k,(l,l+1)} =E^+_{n(k,l), n(k,l+1)}.  
\end{equation*}
We set 
\begin{equation*}
E_{{\mathcal{L}}^{-}} = \bigcup_{(k,l)\in \mbbZ2} E^-_{(k-1,k),l},
\qquad
E_{{\mathcal{L}}^{+}} = \bigcup_{(k,l)\in \mbbZ2} E^+_{k,(l,l+1)},  
\end{equation*}
and $E_\L = E_{\L^-} \cup E_{\L^+}$. 
A two-dimensional configuration 
\begin{equation*}
x =\{ (v_{(k,l)}, e^-_{(k-1,k),l}, e^+_{k,(l,l+1)}) \}_{(k,l) \in \mbbZ2}
\end{equation*}
consists of a vertex 
$v_{(k,l)}\in V_{(k,l)}$ and edges 
$e^-_{(k-1,k),l} \in E_{(k-1,k),l}^-, \, 
 e^+_{k,(l,l+1)} \in E_{k,(l,l+1)}^+
$ 
for each
$(k,l) \in \mbbZ2$
such that  
\begin{gather*}
 s(e^-_{(k-1,k),l}) = v_{(k-1,l)}, \qquad t(e^-_{(k-1,k),l}) = v_{(k,l)}, \\
 s(e^+_{k,(l,l+1)}) = v_{(k,l)}, \qquad t(e^+_{k,(l,l+1)}) = v_{(k,l+1)}.
 \end{gather*}
We further assume that for $(k,l) \in \mbbZ2$,
the label $\lambda^-(e^-_{(k-1,k),l})$ does not depend on $l \in \Z$
as well as  
the label $\lambda^+(e^+_{k,(l,l+1)}) $ does not depend on $k \in \Z$ 
such that there exists $(x_n)_{n\in \Z} \in \Sigma^\Z$ satisfying   
\begin{equation}  
 \lambda^-(e^-_{(k-1,k),l}) = x_k, \qquad
 \lambda^+(e^+_{k,(l,l+1)}) = x_l \quad \text{ for all } (k,l)  \in \mbbZ2,
\label{eq:akal}
\end{equation}
so that the sequence $(x_n)_{n\in \Z}$ yields an element of the presenting subshift
$\Lambda_\L$.
We define $\lambda_\L: E_\L (= E_{\L^-} \cup E_{\L^+}) \longrightarrow \Sigma$ 
by $\lambda_\L = \lambda^- \cup \lambda^+$. 
Since the labeling 
$\lambda^-: E^+\longrightarrow \Sigma$ is left-resolving 
and
$\lambda^-: E^+\longrightarrow \Sigma$ is right-resolving,
the configuration
$x$ is determined by only 
its vertices and labeled sequences on edges.
The configuration $x$ consists of concatenated squares figured as
\begin{equation*} 
\begin{CD}
v_{(k,l)} @>e^+_{k,(l,l+1)}>x_l> v_{(k,l+1)} \\
@A{x_k}A{e^-_{(k-1,k),l}}A @A{x_k}A{e^-_{(k-1,k),l+1}}A \\
v_{(k-1,l)} @>e^+_{k-1,(l,l+1)}>x_l> v_{(k-1,l+1)}. \\
\end{CD}
\end{equation*}
The configuration is figured such as in Figure \ref{fig:conf}.
\begin{figure}[h]
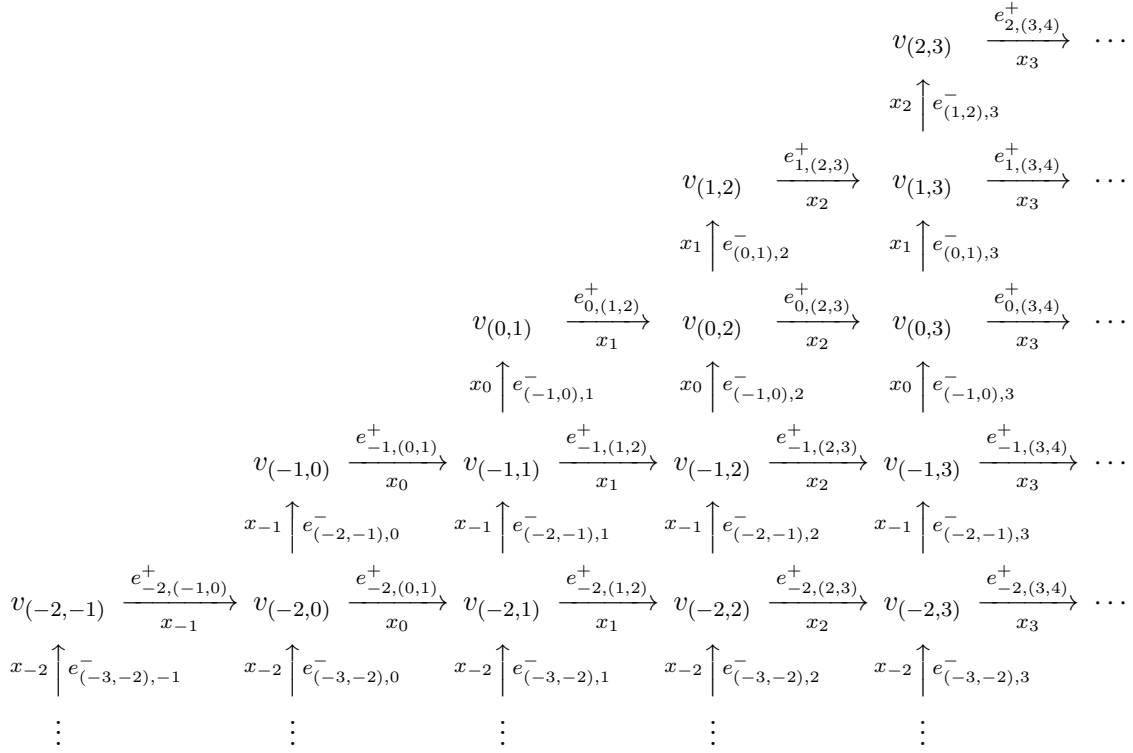

\begin{equation*}
\begin{CD}
@. @. @. @. @. v_{(2,3)} @>{e^+_{2,(3,4)}}>x_3>\cdots\\
@. @. @. @. @. @A{x_{2}}A{e^-_{(1,2),3}}A  @.\\
@. @. @. @. v_{(1,2)} @>{e^+_{1,(2,3)}}>x_2>v_{(1,3)} @>{e^+_{1,(3,4)}}>x_3>\cdots\\
@. @. @. @. @A{x_{1}}A{e^-_{(0,1),2}}A @A{x_{1}}A{e^-_{(0,1),3}}A  @.\\
@. @. @. v_{(0,1)} @>{e^+_{0,(1,2)}}>x_1>v_{(0,2)} @>{e^+_{0,(2,3)}}>x_2>v_{(0,3)} @>{e^+_{0,(3,4)}}>x_3>\cdots \\
@. @. @. @A{x_{0}}A{e^-_{(-1,0),1}}A @A{x_{0}}A{e^-_{(-1,0),2}}A @A{x_{0}}A{e^-_{(-1,0),3}}A  @.\\
@. @. v_{(-1,0)}@>{e^+_{-1,(0,1)}}>x_0>v_{(-1,1)}@>{e^+_{-1,(1,2)}}>x_1>v_{(-1,2)} 
@>{e^+_{-1,(2,3)}}>x_2>v_{(-1,3)} @>{e^+_{-1,(3,4)}}>x_3>\cdots\\
@. @. @A{x_{-1}}A{e^-_{(-2,-1),0}}A @A{x_{-1}}A{e^-_{(-2,-1),1}}A @A{x_{-1}}A{e^-_{(-2,-1),2}}A 
@A{x_{-1}}A{e^-_{(-2,-1),3}}A @.\\
@. v_{(-2,-1)}@>{e^+_{-2,(-1,0)}}>x_{-1}>v_{(-2,0)} @>{e^+_{-2,(0,1)}}>x_0>v_{(-2,1)} 
@>{e^+_{-2,(1,2)}}>x_1>v_{(-2,2)} @>{e^+_{-2,(2,3)}}>x_2>v_{(-2,3)} @>{e^+_{-2,(3,4)}}>x_3>
\cdots \\
@. @A{x_{-2}}A{e^-_{(-3,-2),-1}}A @A{x_{-2}}A{e^-_{(-3,-2),0}}A @A{x_{-2}}A{e^-_{(-3,-2),1}}A 
@A{x_{-2}}A{e^-_{(-3,-2),2}}A @A{x_{-2}}A{e^-_{(-3,-2),3}}A  @.\\
@. \vdots @. \vdots  @. \vdots @. \vdots @.\vdots @. \\
\end{CD}
\end{equation*}
\caption{A configuration}
\label{fig:conf}
\end{figure}
Let us denote by $\XL$ 
the set of two-dimensional configurations on $\mbbZ2$ for the $\lambda$-graph bisystem 
$\L =\LGBS$.
For a given configuration
$x =\{ (v_{(k,l)}, e^-_{(k-1,k),l}, e^+_{k,(l,l+1)} ) \}_{(k,l)\in \mbbZ2}
\in \XL$,
we write the vertices and edges 
$v_{(k,l)}, e^-_{(k-1,k),l}, e^+_{k,(l,l+1)}$ as
$x_{(k,l)}, x^-_{(k-1,k),l}, x^+_{k,(l,l+1)}$, respectively
if we need to specify the given  configuration $x$.
Hence we may write
\begin{equation*}
x =\{ (x_{(k,l)}, x^-_{(k-1,k),l}, x^+_{k,(l,l+1)}) \}_{(k,l)\in \mbbZ2}.
\end{equation*}
When we denote $\lambda^-(e^-_{(k-1,k),l})$ by $x_k$,
and hence  $\lambda^+(e^+_{k,(l,l+1)}) = x_l$, for $(k,l) \in \mbbZ2$,
then the presenting subshift
$\Lambda_\L$ by $\L$ may be written
\begin{equation*}
\Lambda_\L = \{ (x_n)_{n\in \Z} \in \Sigma^\Z \mid x \in \XL \}.
\end{equation*}

 For a configuration 
$x=\{ (v_{(k,l)}, e^-_{(k-1,k),l}, e^+_{k,(l,l+1)} ) \}_{(k,l)\in\mbbZ2} \in \XL$
and a fixed $(p,q) \in \mbbZ2$,
 denote by $\square_{(p,q)}(x)$ the configuration of the lower right infinite rectangle of $x$  for which
its upper left corner is the vertex $v_{(p,q)}$.
It is exactly defined by
\begin{equation}
\square_{(p,q)}(x) =\{ (v_{(k,l)}, e^-_{(k-1,k),l}, e^+_{k,(l,l+1)} ) 
\mid k\le p,  q\le l \}. \label{eq:xpq}
\end{equation}
It is also written 
\begin{equation*}
\square_{(p,q)}(x) =\{ (x_{(k,l)}, x^-_{(k-1,k),l}, x^+_{k,(l,l+1)} ) 
\mid k\le p,  q\le l \}. 
\end{equation*}
Since the labeling 
$\lambda^-: E^+\longrightarrow \Sigma$ is left-resolving 
and
$\lambda^-: E^+\longrightarrow \Sigma$ is right-resolving,
the configuration
$\square_{(p,q)}(x)$ is determined by only 
its vertices and labeled sequences on edges.
By putting
$
x_k =\lambda^-(e^-_{(k-1,k), l}), \, 
x_l =\lambda^+(e^+_{k, (l,l+1)})
$
for
$ 
k \le p, \,\,q\le l,
$
we may write $\square_{(p,q)}(x)$ defined in \eqref{eq:xpq} as
\begin{equation}
\square_{(p,q)}(x) =\{ (v_{(k,l)}, x_k, x_l) \in V\times \Sigma\times \Sigma 
\mid k\le p, q \le l \}. \label{eq:labelxpq}
\end{equation}
We call the configuration $\square_{(p,q)}(x)$ 
the $(p,q)$-{\it rectangle of}\/ $x$.
The $(p,q)$-rectangle $\square_{(p,q)}(x)$ is figured as in Figure \ref{fig:pqrec}.
A $(p,q)$-rectangle expands infinitely in both right and lower directions
from the upper left corner vertex $v_{(p,q)}$.

\begin{figure}[h]
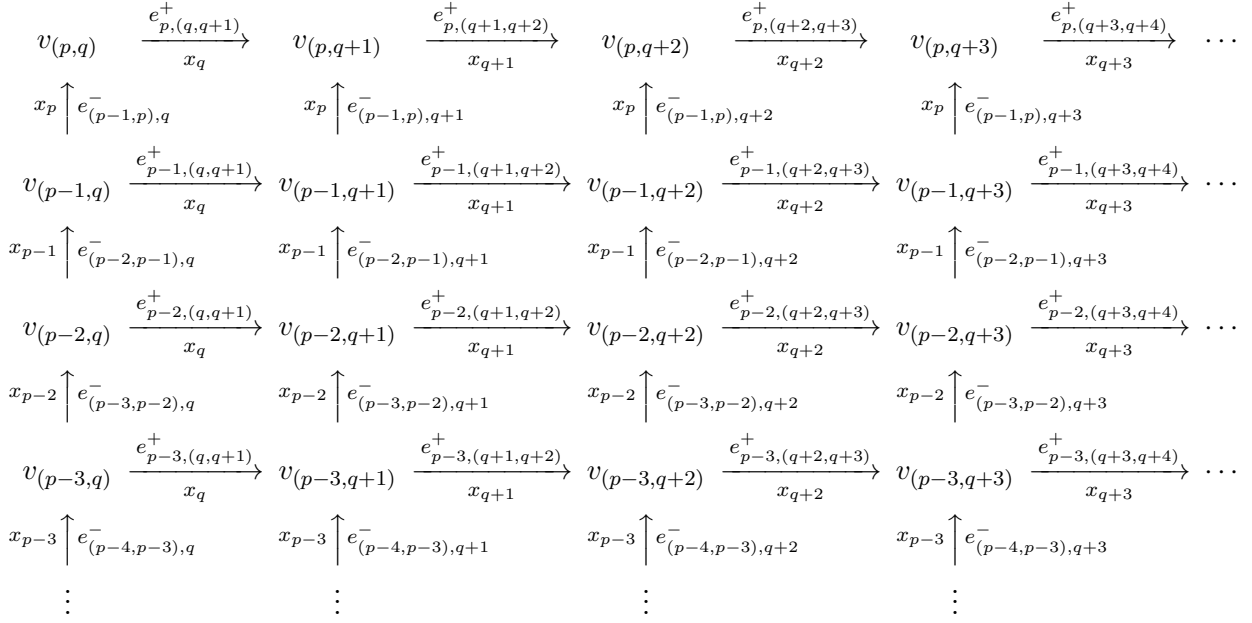

\begin{equation*}
\begin{CD}
v_{(p,q)}@>{e^+_{p,(q,q+1)}}>{x_q}>v_{(p,q+1)} @>{e^+_{p,(q+1,q+2)}}>{x_{q+1}}>v_{(p,q+2)} 
@>{e^+_{p,(q+2,q+3)}}>{x_{q+2}}>v_{(p,q+3)} @>{e^+_{p,(q+3,q+4)}}>{x_{q+3}}>\cdots
\\
@A{x_{p}}A{e^-_{(p-1,p),q}}A @A{x_{p}}A{e^-_{(p-1,p),q+1}}A 
@A{x_{p}}A{e^-_{(p-1,p),q+2}}A @A{x_{p}}A{e^-_{(p-1,p),q+3}}A 
@.\\
v_{(p-1,q)}@>{e^+_{p-1,(q,q+1)}}>{x_q}>v_{(p-1,q+1)}@>{e^+_{p-1,(q+1,q+2)}}>{x_{q+1}}>v_{(p-1,q+2)} 
@>{e^+_{p-1,(q+2,q+3)}}>{x_{q+2}}>v_{(p-1,q+3)} @>{e^+_{p-1,(q+3,q+4)}}>{x_{q+3}}>\cdots
\\
@A{x_{p-1}}A{e^-_{(p-2,p-1),q}}A @A{x_{p-1}}A{e^-_{(p-2,p-1),q+1}}A @A{x_{p-1}}A{e^-_{(p-2,p-1),q+2}}A 
@A{x_{p-1}}A{e^-_{(p-2,p-1),q+3}}A 
@.\\
v_{(p-2,q)} @>{e^+_{p-2,(q,q+1)}}>{x_q}>v_{(p-2,q+1)} 
@>{e^+_{p-2,(q+1,q+2)}}>{x_{q+1}}>v_{(p-2,q+2)} 
@>{e^+_{p-2,(q+2,q+3)}}>{x_{q+2}}>v_{(p-2,q+3)} @>{e^+_{p-2,(q+3,q+4)}}>{x_{q+3}}>\cdots
\\
@A{x_{p-2}}A{e^-_{(p-3,p-2),q}}A @A{x_{p-2}}A{e^-_{(p-3,p-2),q+1}}A 
@A{x_{p-2}}A{e^-_{(p-3,p-2),q+2}}A @A{x_{p-2}}A{e^-_{(p-3,p-2),q+3}}A 
\\
v_{(p-3,q)} @>{e^+_{p-3,(q,q+1)}}>{x_q}>v_{(p-3,q+1)} 
@>{e^+_{p-3,(q+1,q+2)}}>{x_{q+1}}>v_{(p-3,q+2)} 
@>{e^+_{p-3,(q+2,q+3)}}>{x_{q+2}}>v_{(p-3,q+3)} @>{e^+_{p-3,(q+3,q+4)}}>{x_{q+3}}>\cdots
\\
@A{x_{p-3}}A{e^-_{(p-4,p-3),q}}A @A{x_{p-3}}A{e^-_{(p-4,p-3),q+1}}A 
@A{x_{p-3}}A{e^-_{(p-4,p-3),q+2}}A @A{x_{p-3}}A{e^-_{(p-4,p-3),q+3}}A 
@.\\
\vdots @. \vdots @. \vdots @. \vdots @.
\\
\end{CD}
\end{equation*}
\caption{$(p,q)$-rectangle $\square_{(p,q)}(x)$}
\label{fig:pqrec}
\end{figure}


For a vertex $v_{(p,q)} \in V_{(p,q)}(=V_{n(p,q)})$ and a word 
$\mu =(\mu_{p+1}, \mu_{p+2}, \dots, \mu_{q-1}) \in P(v_{(p,q)})$ 
of length $n(p,q)(=q-p-1)$ in $\Lambda_\L$,
recall that $P(v_{(p,q)})$ denote the predecessor set of the vertex 
$v_{(p,q)}$ in  ${\mathcal{L}}^+$.
The word $\mu =(\mu_{p+1}, \mu_{p+2}, \dots, \mu_{q-1}) \in P(v_{(p,q)})$ 
is figured as
\begin{equation*}
\begin{CD}
v_{(p,p+1)}@>{\mu_{p+1}}>>v_{(p,p+2)} @>{\mu_{p+2}}>>\cdots 
@>{\mu_{q-2}}>>v_{(p,q-1)} @>{\mu_{q-1}}>>v_{(p,q)}.
\end{CD}
\end{equation*}
By the left-resolving property of $\L^+$,
there exist  unique finite sequences of vertices and edges in ${\mathcal{L}}^+$
\begin{gather*}
v_{(p,p+1)}\in V_{(p,p+1)}(=V_0), \, 
v_{(p,p+2)}\in V_{(p,p+2)}(=V_1), \, \dots, 
v_{(p,q-1)}\in V_{(p,q-1)}(=V_{n(p,q)-1}),
\\  
e^+_{p,(p+1, p+2)} \in E^+_{p,(p+1, p+2)}(=E^+_{0,1}), \,  
e^+_{p,(p+2,p+3)} \in E^+_{p,(p+2,p+3)}(=E^+_{1,2}), \, \\ \dots, 
e^+_{p,(q-1,q)} \in E^+_{p,(q-1,q)}(=E^+_{n(p,q)-1, n(p,q)})
\end{gather*}
such that 
\begin{gather*}
v_{(p,p+1)} = s(e^+_{p,(p+1, p+2)}), \, 
v_{(p,p+2)} =t(e^+_{p,(p+1, p+2)}) = s(e^+_{p,(p+2, p+3)}), \,
\dots,\\
 v_{(p,q-1)} = s(e^+_{p,(q-1, q)}), \,
 v_{(p,q)} = t(e^+_{p,(q-1, q)}), \\
\mu_{p+1} = \lambda^+(e^+_{p,(p+1, p+2)}), \,  
\mu_{p+2}= \lambda^+(e^+_{p,(p+2,p+3)}), \dots,
\mu_{q-1} =\lambda^+(e^+_{p,(q-1,q)}).
\end{gather*}
Since the $\lambda$-graph bisystem $\LGBS$ satisfies FPCC,
we have 
$P(v_{(p,q)}) = F(v_{(p,q)})$
so that, by the right-resolving property of $\L^-$, 
there exist  unique finite sequences of vertices and edges in ${\mathcal{L}}^-$
\begin{gather*}
v_{(q-1,q)}\in V_{(q-1,q)}(=V_0), \, 
v_{(q-2,q)}\in V_{(q-2,q)}(=V_1), \, \dots, 
v_{(p-1,q)}\in V_{(p-1,q)}(=V_{n(p,q)-1}),
\\  
e^-_{(q-2,q-1),q} \in E^-_{(q-2,q-1),q}(=E^-_{1,0}), \,  
e^-_{(q-3,q-2),q} \in E^-_{(q-3,q-2),q}(=E^-_{2,1}), \, \\ \dots,
e^-_{(p,p-1),q} \in E^-_{(p,p-1),q}(=E^-_{n(p,q), n(p,q)-1})
\end{gather*}
such that 
\begin{gather*}
v_{(q-1,q)} = t(e^-_{(q-2, q-1),q}), \, 
v_{(q-2,q)} =s(e^-_{(q-2, q-1),q}) = t(e^-_{(q-3,q-2),q}), \,
\dots,\\
 v_{(p-1,q)} = t(e^-_{(p,p-1),q}), \,
 v_{(p,q)} = s(e^-_{(p,p-1),q}), \\
\mu_{p+1} = \lambda^+(e^-_{(p,p-1),q}), \,  
\mu_{p+2}= \lambda^+(e^-_{(p-1,p-2),q}), \dots,
\mu_{q-1} =\lambda^+(e^-_{(q-2,q-1),q}).
\end{gather*}
The situation is figured such as in Figure \ref{fig:vpqmu}.
\begin{figure}[h]
\begin{equation*}
\begin{CD}
  @. @. @. v_{(q-1,q)} \\
  @. @. @. @A{\mu_{q-1}}A{e^-_{(q-2,q-1),q}}A @.\\
  @. @. @. \vdots \\
  @. @. @. @A{\mu_{p+2}}A{e^-_{(p+1,p+2),q}}A @.\\
  @. @. @. v_{(p+1,q)} \\
  @. @. @. @A{\mu_{p+1}}A{e^-_{(p,p+1),q}}A @.\\
 v_{(p,p+1)} @>{e^+_{p,(p+1,p+2)}}>{\mu_{p+1}}> v_{(p,p+2)} @>{e^+_{p,(p+2,p+3)}}>{\mu_{p+2}}>
\cdots 
@>{e^+_{p,(q-1,q)}}>{\mu_{q-1}}>v_{(p,q)}\\
\end{CD}
\end{equation*}
\label{fig:vpqmu}
\caption{$P(v_{(p,q)})$ and $F(v_{(p,q)})$}
\end{figure} 
By the property \eqref{eq:akal}, 
we may file the configuration
 $\{ (v_{(k,l)}, e^-_{(k-1,k),l}, e^+_{k,(l,l+1)}) \mid p \le k < l \le q \}$
from a given vertex
$v_{(p,q)} \in V_{(p,q)}$ and a word
$\mu = (\mu_{p+1}, \dots,\mu_{q-1}) \in P(v_{(p,q)})$. 
The resulting configuration is denoted by
$\triangle_{(\mu;v_{(p,q)})}$,
that is figured in Figure \ref{fig:pqtri}.
 For a fixed $(p,q) \in \mbbZ2$,
let us denote by $\triangle_{(p,q)}(x) $ the configuration of
$x$ restricting to $\{(k,l) \in \mbbZ2 \mid p\le k < l \le q\}$.
It is written
\begin{equation}
\triangle_{(p,q)}(x) =\{ (x_{(k,l)}, x^-_{(k,k+1),l}, x^+_{k,(l,l+1)}) 
\mid p\le k < l \le q \}, \label{eq:tpq}
\end{equation}
and called the $(p,q)$-{\it triangle for}\/ $x$.
\begin{figure}[h]
{\small
\begin{equation*}
\begin{CD}
 @. @. @. @. v_{(q-1,q)} \\
 @. @. @. @. @A{\mu_{q-1}}A{e^-_{(q-2,q-1),q}}A @.\\
@. @. @. v_{(q-2,q-1)} @>{e^+_{q-2,(q-1,q)}}>{\mu_{q-1}}> v_{(q-2,q)} \\ 
@. @. @. @A{\mu_{q-2}}A{e^-_{(q-3,q-2),q-1}}A  @A{\mu_{q-2}}A{e^-_{(q-3,q-2),q}}A @.\\
@. @. v_{(p+2,p+3)} @>{e^+_{p+2,(p+3,p+4)}}>{\mu_{p+3}}>\cdots @. \vdots \\
 @. @. @A{\mu_{p+2}}A{e^-_{(p+1,p+2),p+3}}A @. @A{\mu_{p+2}}A{e^-_{(p+1,p+2),q}}A @.\\
 @. v_{(p+1,p+2)} @>{e^+_{p+1,(p+2,p+3)}}>{\mu_{p+2}}>
v_{(p+1,p+3)} @>{e^+_{p+1,(p+3,p+4)}}>{\mu_{p+3}}>
\cdots
 @>{e^+_{p+1,(q-1,q)}}>{\mu_{q-1}}> v_{(p+1,q)} \\
 @. @A{\mu_{p+1}}A{e^-_{(p,p+1),p+2}}A @A{\mu_{p+1}}A{e^-_{(p,p+1),p+3}}A @.  
@A{\mu_{p+1}}A{e^-_{(p,p+1),q}}A @.\\
 v_{(p,p+1)} @>{e^+_{p,(p+1,p+2)}}>{\mu_{p+1}}>v_{(p,p+2)} @>{e^+_{p,(p+2,p+3)}}>{\mu_{p+2}}>
 v_{(p,p+3)} @>{e^+_{p,(p+3,p+4)}}>{\mu_{p+3}}>
\cdots  @>{e^+_{p,(q-1,q)}}>{\mu_{q-1}}>v_{(p,q)}\\
\end{CD}
\end{equation*}
}
\caption{$(p,q)$-triangle $\triangle_{(p,q)}$}
\label{fig:pqtri}
\end{figure}
 \begin{lemma}\label{lem:pqtriangle}
A $(p,q)$-triangle is exactly determined by 
the vertex $v_{(p,q)}$ and the word $\mu \in P(v_{(p,q)}).$
Hence there exists a bijective correspondence between
the set of $(p,q)$-triangles and the set of words
$P(v_{(p,q)})$ with vertices $v_{(p,q)}$ in $V_{(p,q)}$.
\end{lemma}
As in the above lemma, the vertex  
$v_{(p,q)}$ and the word $\mu = (\mu_{p+1},\dots,\mu_{q-1})\in P(v_{(p,q)})$ 
determine the original $(p,q)$-triangle
for a given configuration $x \in \XL$.
\medskip

We define the cylinder set 
$U_{(\mu;v_{(p,q)})}$
in $\XL$ for 
$v_{(p,q)} \in V_{(p,q)}$  and
$\mu =(\mu_{p+1}, \dots, \mu_{q-1}) \in P(v_{(p,q)})$
by setting
\begin{equation*}
U_{(\mu;v_{(p,q)})} =
\{ x \in \XL \mid \triangle_{(p,q)}(x) = \triangle_{(\mu;v_{(p,q)})} \}.
\end{equation*}
For a configuration $x =\{ (v_{(k,l)}, e^-_{(k-1,k),l}, e^+_{k,(l,l+1)} ) \}_{(k,l)\in\mbbZ2}$
and a vertex $v_{(p,q)}(=x_{(p,q)})$ at $(p,q) \in \mbbZ2$ in $x$,
define admissible words 
$F_x(v_{(p,q)}), P_x(v_{(p,q)})$ 
of $\Lambda_\L$ by 
\begin{align*}
F_x(v_{(p,q)}) & = (\lambda^-(e^-_{(p,p+1),q}),\lambda^-(e^-_{(p+1,p+2),q}),\dots,
\lambda^-(e^-_{(q-2,q-1),q})) \in B_{n(p,q)}(\Lambda_\L), \\
P_x(v_{(p,q)})& = (\lambda^+(e^+_{p,(p+1,p+2)}),\dots,
\lambda^+(e^+_{p,(q-2,q-1)}), \lambda^+(e^+_{p,(q-1,q)}) ) \in B_{n(p,q)}(\Lambda_\L).
 \end{align*}
By the condition \eqref{eq:akal}, we know that
\begin{equation*}
F_x(v_{(p,l)}) = P_x(v_{(p,q)}) \quad \text{ for all } (p,q) \in \mbbZ2.
\end{equation*}
As the triangle
$\triangle_{(\mu;v_{(p,q)})}$ is determined by only 
the vertex 
$v_{(p,q)} \in V_{(p,q)}$ and
the word
$\mu  \in P(v_{(p,q)})$,
we know that 
$x \in \XL$ belongs to
$U_{(\mu;v_{(p,q)})}$ if and only if
$x_{(p,q)} = v_{(p,q)},  P_x(x_{(p,q)})=\mu.
$
\begin{lemma}
The set of the form
$\{ U_{(\mu;v_{(p,q)})} \mid
v_{(p,q)} \in V_{n(p,q)}, \mu \in P(v_{(p,q)}) \}$
becomes an open neighborhood basis in $\XL$.
 \end{lemma}
\begin{proof}
Let $\mu\in P(v_{(p,q)})$ and  $\mu'\in P(v'_{(p',q')})$,
so that $p<q$ and $p'<q'$.
Suppose that $U_{(\mu;v_{(p,q)})}\cap U_{(\mu';v'_{(p',q')})} \ne \emptyset$.
Take a configuration 
$x \in U_{(\mu;v_{(p,q)})}\cap U_{(\mu';v'_{(p',q')})}$
so that
$x_{(p,q)}= v_{(p,q)}, x_{(p',q')} = v'_{(p',q')}$
and
$P_x(v_{(p,q)}) = \mu,
P_x(v'_{(p',q')}) = \mu'$.  
Let
$p_0 := \min\{p,p'\}$ and 
$q_0 := \max\{q,q'\}$,
so that $p_0< q_0$.
Put $\nu = P_x(x_{(p_0,q_0)})$.
Since 
both triangles $\triangle_{(\mu;v_{(p,q)})}$
and $\triangle_{(\mu';v'_{(p',q')})}$
are parts of the triangle
$\triangle_{(\nu;v_{(p_0,q_0)})}$,
we have
$x \in U_{(\nu;v_{(p_0,q_0)})}\subset U_{(\mu;v_{(p,q)})}\cap U_{(\mu';v'_{(p',q')})}$.
\end{proof}
We endow the topology 
with  $\XL$
for which open neighborhood basis are of the form
 $\{U_{(\mu;v_{(p,q)})} \mid
v_{(p,q)} \in V_{(p,q)}, \mu \in P(v_{(p,q)}) \}$.
 
\begin{proposition}
The set $\XL$ becomes a compact Hausdorff space by the topology defined above.
\end{proposition}
\begin{proof}
We will first show that $\XL$ is Hausdorff.
Let $x, x' \in \XL$ be  two configurations such that $x\ne x'$.
There exists $(p,q) \in \mbbZ2$ such that 
$P_x(x_{(p,q)})\ne P_{x'}(x'_{(p,q)})$.
By putting
$v_{(p,q)} = x_{(p,q)},  \mu = P_x(x_{(p,q)})$
and
$v'_{(p,q)} = x'_{(p,q)},  \mu' = P_{x'}(x'_{(p,q)})$,
we have
$\triangle_{(\mu;v_{(p,q)})} \ne \triangle_{(\mu';v'_{(p,q)})}$
and hence
$U_{(\mu;v_{(p,q)})} \cap U_{(\mu';v'_{(p,q)})} =\emptyset$.

We will next show that 
$\XL$ is compact.
We will see in Lemma \ref{lem:metric}
that $\XL$ is a metric space. 
Hence
it suffices to show that $\XL$ is sequentially compact.
Let $x(n), n\in \N$ be a sequence of configurations in  $\XL$.
Let $\triangle_{(p,q)}(x(n))$ 
be the $(p,q)$-triangle of the configuration $x(n)$. 
For a fixed $(p,q)$, there are only finitely many  
$(p,q)$-triangles, because the vertex set 
$V_{(p,q)}$ and the edges sets 
$
\cup_{p\le k<l \le q} ( E_{(k,k+1),l}^- \cup E_{k,(l,l+1)}^+ )
(= \cup_{l=0}^{n(p,q)-1}(E_{l,l+1}^-\cup E^+_{l,l+1}))
 $ 
are finite.
Hence there is an infinite set $I_1$ of natural numbers for which  
 $\triangle_{(-1,1)}(x(n))$ 
is the same for all $n \in I_1$.
Next, we may find an infinite subset $I_2 \subset I_1$ such that  
$\triangle_{(-2,2)}(x(n))$ is the same for all $n \in I_2$.
Continuing this procedure, we may find for each $m\in \N$ 
an infinite set $I_m \subset I_{m-1}$ such that 
$\triangle_{(-m,m)}(x(n))$ is the same for all $n \in I_m$.
We may find a configuration  $x$ belonging to $\XL$
such that $\triangle_{(-m,m)}(x) =\triangle_{(-m,m)}(x(n))$ for all $n\in I_m$.
By the choice of axiom, one may take $i_n \in I_n$ for all $n \in \N$.
We then know that the subsequence
$x(i_n), n \in \N$ of  $x(n), n\in \N$ converges to $x$ 
in the topology.
This shows that  
$\XL$ is compact.
\end{proof}

\medskip

We will define a metric on the configuration space  
$\XL$ that defines the same topology previously endowed.
For two configurations
$x,z \in \XL$,
if $\triangle_{(-1,1)}(x) \ne \triangle_{(-1,1)}(z)$,
then we define 
$d(x,z) =1$.
If there exists $p \in \N$ such that 
\begin{equation*}
\triangle_{(-p,p)}(x) = \triangle_{(-p,p)}(z), \qquad
\triangle_{(-p-1,p+1)}(x) \ne \triangle_{(-p-1,p+1)}(z),
\end{equation*}
we define
$
d(x,z) =\frac{1}{2^p}.
$
Then we note that 
\begin{equation*}
d(x,z) \le \frac{1}{2^p} \quad \text{ if and only if } 
\quad \triangle_{(-p,p)}(x) = \triangle_{(-p,p)}(z).
\end{equation*}
It is straightforward to see that the function
$d(\,\, , \,\, ): \XL\times\XL\longrightarrow \mathbb{R}$
yields a metric on $\XL$ that gives rise to the same topology on $\XL$
defined previously.  
 We may see more about the metric in the following. 
\begin{lemma}\label{lem:metric}
The metric 
$d(\,\, , \,\, )$ is an ultrametric on $\XL$, that means the inequality
\begin{equation*}
d(x,y) \le \max\{d(x,z), d(z,y) \}
\text{ for all } x,y,z \in \XL
\end{equation*}
holds.
\end{lemma}
\begin{proof}
Suppose that 
$d(x, y) = \frac{1}{2^p}.$
Hence we have 
\begin{equation*}
\triangle_{(-p,p)}(x) = \triangle_{(-p,p)}(y), \qquad 
\triangle_{(-p-1,p+1)}(x) \ne \triangle_{(-p-1,p+1)}(y).
\end{equation*}
If both the equalities
\begin{equation*}
\triangle_{(-p-1,p+1)}(x) = \triangle_{(-p-1,p+1)}(z), \qquad 
\triangle_{(-p-1,p+1)}(z) = \triangle_{(-p-1,p+1)}(y)
\end{equation*}
hold, then
$\triangle_{(-p-1,p+1)}(z) \ne \triangle_{(-p-1,p+1)}(y)$,
so that  
$
\triangle_{(-p-1,p+1)}(x) \ne \triangle_{(-p-1,p+1)}(z).
$
Hence we have
$d(x, z) \ge \frac{1}{2^p}$
proving
\begin{equation*}
d(x,y)=\frac{1}{2^p} \le \max\{d(x,z), d(z,y) \}.
\end{equation*}
\end{proof}
Since a compact metric space defined by an ultrametric is totally disconnected
(cf. \cite[Theorem 2.10]{PutnamAMS}),
we have 
\begin{proposition}
The set $\XL$ is a compact
totally disconnected metric space.
\end{proposition}

There is a possiblity that the  compact totally disconnected metric space
$\XL$ has isolated points. 
We need an assumption on $\LGBS$ under which
$\XL$ is a Cantor set.
We will discuss a condition for the $\lambda$-graph bisystem
$\LGBS$ that 
the space $\XL$ is a Cantor set in the following section.


\section{Realization of the configuration space}
In what follows, 
we fix a  $\lambda$-graph bisystem $\L=\LGBS$ satisfying FPCC. 
The following lemma is crucial in our further discussion
\begin{lemma}\label{lem:muvx}
Take a vertex $v_{(p,q)} \in V_{(p,q)}$
with $(p,q) \in \mbbZ2$.
For a $(p,q)$-rectangle $\square_{(p,q)}$ and a finite word
$\mu \in P(v_{(p,q)})$,
there exists a two-dimensional configuration $x \in \XL$
such that 
 $\square_{(p,q)}(x) =\square_{(p,q)}$ and $ P_x(v_{(p,q)}) =\mu $. 
The two-dimensional configuration $x$ is unique for 
$\square_{(p,q)}$ and $\mu \in P(v_{(p,q)})$. 
\end{lemma}
\begin{proof}
Let $\square_{(p,q)}$ be a $(p,q)$-rectangle such that 
\begin{equation}
\square_{(p,q)} =\{ (v_{(k,l)}, e^-_{(k-1,k),l}, e^+_{k,(l,l+1)} ) 
\mid k\le p,\, q \le l \}. \label{eq:recpq}
\end{equation}
Put
\begin{equation*}
x_k  = \lambda^-(e^-_{(k-1,k),l}) \quad \text{ for } k\le p, \quad \text{ and } \quad
x_l  = \lambda^+(e^+_{k,(l,l+1)}) \quad \text{ for } q\le l.
\end{equation*}
We write the word $\mu\in P(v_{(p,q)})$ as
$
\mu = (x_{p+1}, x_{p+2},\dots, x_{q-1}),
$
so that we have a biinfinite sequence $(x_n)_{n\in\Z} \in \Sigma^\Z$
that belongs to $\Lambda_\L$.
By the left-resolving property of ${\mathcal{L}}^+$,
there exists a unique finite path of directed edges 
\begin{equation*}
e^+_{p,(l,l+1)}\in E^+_{p,(l,l+1)}, \qquad l=p+1,\dots,q-1
\end{equation*}
in $\L^+$ such that  
\begin{gather*}
t(e^+_{p,(l,l+1)})  = s(e^+_{p,(l+1,l+2)}) \quad \text{ for } l=p+1,\dots,q-2,\qquad
t(e^+_{p,(q-1,q)})  = v_{(p,q)},\\
\lambda^+(e^+_{p,(l,l+1)})  = x_l\quad \text{ for } l=p+1,\dots, q-1. 
\end{gather*}
We put
\begin{equation*}
 v_{(p,l)} = s(e^+_{p,(l,l+1)}) \in V_{(p,l)} \quad \text{ for } l=p+1,\dots,q-1 
\end{equation*}
so that we have a labeled path:
\begin{equation*}
\begin{CD}
v_{(p,p+1)}@>{e^+_{p,(p+1,p+2)}}>{x_{p+1}}>v_{(p,p+2)} 
@>{e^+_{p,(p+2,p+3)}}>{x_{p+2}}>\cdots 
@>{e^+_{p,(q-2,q-1)}}>{x_{q-2}}>v_{(p,q-1)} 
@>{e^+_{p,(q-1,q)}}>{x_{q-1}}>v_{(p,q)}.
\end{CD}
\end{equation*}
By the local property of $\lambda$-graph bisystem $\LGBS$,
the given directed edge 
$e^-_{(p-1,p),q}$ in $\square_{(p,q)}$
together with  
$e^+_{p,(q-1,q)}$  satisfying
$$
t(e^-_{(p-1,p),q}) = v_{(p,q)} =t(e^+_{p,(q-1,q)}), \qquad
s(e^-_{(p-1,p),q}) = v_{(p-1,q)}
$$
determine directed edges 
$
e^-_{(p-1,p),q-1} \in E^-_{(p-1,p),q-1}
$ 
and 
$
e^+_{p-1,(q-1,q)}\in E^+_{p-1, (q-1,q)}
$
such that 
\begin{gather*}
t(e^-_{(p-1,p),q-1}) =v_{(p,q-1)},  \qquad
s(e^-_{(p-1,p),q-1}) =s(e^+_{p-1,(q-1,q)}), \qquad
t(e^+_{p-1,(q-1,q)}) = v_{(p-1,q)},\\
\lambda^-(e^-_{(p-1,p),q-1}) =x_p,\qquad
\lambda^+(e^+_{p-1,(q-1,q)}) = x_{q-1}, \qquad
t(e^+_{p-1,(q-1,q)}) = v_{(p-1,q)}
\end{gather*}
that is figured such as 
\begin{equation*} 
\begin{CD}
v_{(p,q-1)} @>e^+_{p,(q-1,q)}>{x_{q-1}}> v_{(p,q)} \\
@A{x_p}A{e^-_{(p-1,p),q-1}}A @A{x_p}A{e^-_{(p-1,p),q}}A \\
v_{(p-1,q-1)} @>e^+_{p-1,(q-1,q)}>{x_{q-1}}> v_{(p-1,q)}. \\
\end{CD}
\end{equation*}
Like this way, by using the local property of $\L$,
we may extend the $(p,q)$-rectangle $\square_{(p,q)}$
to a $(p,p+1)$-rectangle $\square_{(p,p+1)}$ in a unique way.
We write it as
\begin{equation}
\square_{(p,p+1)} =\{ (v_{(k,l)}, e^-_{(k-1,k),l}, e^+_{k,(l,l+1)} )
\mid k\le p,\, p+1 \le l \}. \label{eq:recpp1}
\end{equation}

Similarly by the right-resolving property of ${\mathcal{L}}^-$,
there exists a unique finite path of edges in ${\mathcal{L}}^-$
\begin{equation*}
e^-_{(k-1,k),q}\in E^-_{(k-1,k),q}, \qquad k=p+1,\dots,q-1 
\end{equation*}
such that 
\begin{gather*}
s(e^-_{(p,p+1),q}) = v_{(p,q)},\qquad
t(e^-_{(k-1,k),q}) = s(e^-_{(k,k+1),q}) \quad \text{ for } k=p+1,\dots,q-2,\\
\lambda^-(e^-_{(k-1,k),q}) =  x_k \quad \text{ for } k=p+1,\dots,q-1.
\end{gather*}
We set
\begin{equation*}
v_{(k,q)} = t(e^-_{(k-1,k),q}) \in V_{(k,q)}, \qquad k=p+1,\dots,q-1.
\end{equation*}
The situation is figured such as 
\begin{equation*}
\begin{CD}
 v_{(q-1,q)} \\
 @A{x_{q-1}}A{e^-_{(q-2,q-1),q}}A @.\\
 \vdots \\
 @A{x_{p+2}}A{e^-_{(p+1,p+2),q}}A @.\\
 v_{(p+1,q)} \\
 @A{x_{p+1}}A{e^-_{(p,p+1),q}}A @.\\
v_{(p,q)}\\
\end{CD}
\end{equation*}
By the local property of $\L$, 
we may extend the $(p,q)$-rectangle $\square_{(p,q)}$ to 
a ${(q-1,q)}$-rectangle $\square_{(q-1,q)}$
 in a similar way to the extension from
$\square_{(p,q)}$ to $\square_{(p,p+1)}$.
We write the $\square_{(q-1,q)}$-rectangle as 
\begin{equation}
\square_{(q-1,q)} =\{ (v_{(k,l)}, e^-_{(k-1,k),l}, e^+_{k,(l,l+1)}) 
\mid k\le q-1,\, q \le l \}. \label{eq:recq1q}
\end{equation}
Therefore we have extensions
$
\square_{(p,p+1)}\cup
\square_{(q-1,q)}$
of
$\square_{(p,q)}$.

We will next extend $\square_{(p,p+1)}$ to its lower left.
By the left-resolving property of ${\mathcal{L}}^+$,
there exists a unique edge written $e^+_{p-1,(p,p+1)} \in E^+_{p-1,(p,p+1)}$
such that 
\begin{equation*}
t(e^+_{p-1,(p,p+1)}) = v_{(p-1,p+1)}, \qquad \lambda^+(e^+_{p-1,(p,p+1)}) = x_p.
\end{equation*} 
We put
$
v_{(p-1,p)} =s(e^+_{p-1,(p,p+1)}) \in V_{(p-1,p)}.
$ 
By the local property of $\L$ for the vertices
$v_{(p-1,p)}$ and $v_{(p-2,p+1)}$,
we may find unique directed edges and a vertex 
$$
e^-_{(p-2,p-1),p} \in E^-_{(p-2,p-1),p}
\qquad
e^+_{p-2,(p,p+1)}\in E^+_{p-2,(p, p+1)}
\quad
\text{ and }
\quad
v_{(p-2, p)}\in V_{(p-2,p)}
 $$
such that 
\begin{gather*}
t(e^-_{(p-2,p-1),p}) =v_{(p-1,p)},  \qquad
s(e^-_{(p-2,p-1),p}) =s(e^+_{p-2,(p,p+1)}) =v_{(p-2, p)}, \\
t(e^+_{p-2,(p,p+1)}) = v_{(p-2,p+1)},\qquad
\lambda^-(e^-_{(p-2,p-1),p}) = x_{p-1},\qquad
\lambda^+(e^+_{p-2,(p,p+1)}) = x_p,
\end{gather*}
that is figured such as 
\begin{equation*} 
\begin{CD}
v_{(p-1,p)} @>e^+_{p-1,(p,p+1)}>x_p> v_{(p-1,p+1)} \\
@A{x_{p-1}}A{e^-_{(p-2,p-1),p}}A @A{x_{p-1}}A{e^-_{(p-2,p-1),p+1}}A \\
v_{(p-2,p)} @>e^+_{p-2,(p,p+1)}>x_p> v_{(p-2,p+1)}. \\
\end{CD}
\end{equation*}
Like this way, we may extend $\square_{(p,p+1)}$ to its one-step left
so that we have the extension $\square_{(p-1,p)}$.
By continuing these procedure, 
we have successive extensions
 $\square_{(p-j,p-j+1)}$
from $\square_{(p-j +1,p-j+2)}$
for $j=2,3,\dots $.
Similarly, we have an extension $\square_{(q,q+1)}$ 
from $\square_{(q-1,q)}$
and a sequence of extensions
$\square_{(q+j,q+j+1)}$
from $\square_{(q+j-1,q+j)}$
for $j=2,3,\dots $.
Therefore we have an extension 
of $\square_{(p,q)}$
to a configuration except the inside the 
$(p,q)$-triangle.
Now the word 
$\mu =(x_{p+1}, x_{p+2},\dots, x_{q-1}) \in P(v_{(p,q)})$
is given,
so that Lemma \ref{lem:pqtriangle} tells us that 
it uniquely extends to $(p,q)$-triangle $\triangle_{(\mu;v_{(p,q)})}$.
By the right-resolving property of ${\mathcal{L}}^-$, 
the directed edges in ${\mathcal{L}}^-$ whose source is $v_{(p,q)}$ and its label sequence 
is $\mu$ is unique,
so that the $(p,q)$-tritangle $\triangle_{(\mu;v_{(p,q)})}$ 
files the inside the above extension of 
$\square_{(p,q)}$.
Therefore we get a whole configuration from $\square_{(p,q)}$ and 
$\mu  \in P(v_{(p,q)})$.
The configuration is clearly unique by the left-resolving property of ${\mathcal{L}}^+$ and
the right-resolving property of ${\mathcal{L}}^-$.
\end{proof}

For $(p,q) \in \mbbZ2$,
we fix a vertex $v_{(p,q)}\in V_{(p,q)}$.
A {\it zigzag path}\/ $\gamma$ starting from $v_{(p,q)}$ is a concatenating infinite path
$$
\gamma=(e^+_{p-k,(q+k,q+k+1)},e^-_{(p-k-1,p-k),q+k+1})_{k=0}^\infty
$$
such that
\begin{gather*}
s(e^+_{p,(q,q+1)}) = v_{(p,q)}, \qquad
t(e^+_{p-k,(q+k,q+k+1)}) = t(e^-_{(p-k-1,p-k),q+k+1}),\\
s(e^-_{(p-k-1,p-k),q+k+1}) =s(e^+_{p-k,(q+k,q+k+1)}), \qquad
k=0,1,\dots
\end{gather*}
Since the labeling 
$\lambda^-: E^+\longrightarrow \Sigma$ is left-resolving and
and
$\lambda^-: E^+\longrightarrow \Sigma$ is right-resolving,
the configuration
 is determined by only 
its vertices and labeled sequences on edges.
Let
\begin{equation*}
x_k =\lambda^-(e^-_{(k-1,k), l}), \qquad 
x_l =\lambda^+(e^+_{k, (l,l+1)}) \quad \text{ for } k \le p< q \le l.
\end{equation*}
A zigzag path starting from $v_{(p,q)}$ is figured such as 
in Figure \ref{fig:zigzag1}.
\begin{figure}[h]
\begin{equation*}
\begin{CD}
v_{(p,q)}@>{e^+_{p,(q,q+1)}}>x_q>v_{(p,q+1)} @. 
@. @.
\\
@. @A{x_p}A{e^-_{(p-1,p),q+1}}A @. @. 
@.\\
@. v_{(p-1,q+1)}@>{e^+_{p-1,(q+1,q+2)}}>{x_{q+1}}>v_{(p-1,q+2)} 
@. @.
\\
@. @. @A{x_{p-1}}A{e^-_{(p-2,p-1),q+2}}A 
@. 
@.\\
@. 
@. v_{(p-2,q+2)} @>{e^+_{p-2,(q+2,q+3)}}>{x_{q+2}}>v_{(p-2,q+3)} @.
\\
@. @. 
@. @A{x_{p-2}}A{e^-_{(p-3,p-2),q+3}}A 
\\
@. @. @. \cdots  @. 
\\
\\\end{CD}
\end{equation*}
\caption{A zigzag path}
\label{fig:zigzag1}
\end{figure}
We denote by 
$\CZ(v_{(p,q)})$
the set of zigzag paths starting from a vertex $v_{(p,q)}$
in $V_{(p,q)}$,
and set
$$
\CZ_{(p,q)} = \bigcup_{v_{(p,q)} \in V_{(p,q)}} \CZ(v_{(p,q)}).
$$
\begin{lemma}
For a vertex  $v_{(p,q)} \in V_{(p,q)}$,
a zigzag path starting from 
 $v_{(p,q)}$ determines a $(p,q)$-rectangle as its extension in a unique way.
Hence there exists a bijective correspondence between zigzag paths  
starting from  $v_{(p,q)}$
and $(p,q)$-rectangles $\square_{(p,q)}$
whose upper left corner is $v_{(p,q)}$.
\end{lemma}
\begin{proof}
Let
$\gamma =
(e^+_{p-k,(q+k,q+k+1)},e^-_{(p-k-1,p-k),q+k+1})_{k=0}^\infty
\in \CZ(v_{(p,q)}).
$
We put
\begin{align*}
v_{(p-k,q+k+1)} = & t(e^+_{p-k,(q+k,q+k+1)}) ( =t(e^-_{(p-k-1,p-k),q+k+1})), \\
v_{(p-k,q+k)} = & s(e^+_{p-k,(q+k,q+k+1)}) ( = s(e^-_{(p-k,p-k+1),q+k})).
\end{align*}
By the local property of $\lambda$-graph bisystem,
there exists a vertex
$v_{(p-k-1, q+k)} \in V_{(p-k-1, q+k)}$
and edges 
$e^+_{p-k-1,(q+k,q+k+1)},\, 
 e^-_{(p-k-1,p-k),q+k}
$
such that 
\begin{gather*}
\lambda^+(e^+_{p-k-1,(q+k,q+k+1)}) = \lambda^+(e^+_{p-k,(q+k,q+k+1)}), \\
\lambda^-(e^-_{(p-k-1,p-k),q+k}) = \lambda^-(  e^-_{(p-k-1,p-k),q+k+1}), \\
s(e^+_{p-k-1,(q+k,q+k+1)}) = s( e^-_{(p-k-1,p-k),q+k}) = v_{(p-k-1,q+k)}, \\
t(e^+_{p-k,(q+k,q+k+1)})= v_{(p-k,q+k+1)},
\qquad 
t( e^-_{(p-k-1,p-k),q+k}) = v_{(p-k,q+k)}
\end{gather*}
for all $k \in \Zp$
that are figured such as 
\begin{equation*} 
\begin{CD}
v_{(p-k, q+k)} @>e^+_{p-k,(q+k,q+k+1)}>> v_{(p-k,q+k+1)} \\
@AA{e^-_{(p-k-1,p-k),q+k}}A @AA{e^-_{(p-k-1,p-k),q+k+1}}A \\
v_{(p-k-1,q+k)} @>e^+_{p-k-1,(q+k,q+k+1)}>> v_{(p-k-1,q+k+1)}. \\
\end{CD} 
\end{equation*}
Similarly by the local property of $\lambda$-graph bisystem,
there exists a vertex
$v_{(p-k, q+k+2)} \in V_{(p-k, q+k+2)}$
and edges 
$e^+_{p-k,(q+k+1,q+k+2)},\, 
 e^-_{(p-k-1,p-k),q+k+2)}
$
such that 
\begin{gather*}
\lambda^+(e^+_{p-k,(q+k+1,q+k+2)})  = \lambda^+(e^+_{p-k-1,(q+k+1,q+k+2)}), \\
\lambda^-( e^-_{(p-k-1,p-k),q+k+2}) = \lambda^-( e^-_{(p-k-1,p-k),q+k+1}),\\
t(e^+_{p-k,(q+k+1,q+k+2)})  = t(e^-_{(p-k-1,p-k),q+k+2}) = v_{(p-k,q+k+2)},\\
s(e^+_{p-k,(q+k+1,q+k+2)})  = v_{(p-k,q+k+1)}, \qquad
s( e^-_{(p-k-1,p-k),q+k+2}) = v_{(p-k-1,q+k+2)}
\end{gather*}
for all $k \in \Zp$
that are figured such as 
\begin{equation*} 
\begin{CD}
v_{(p-k, q+k+1)} @>e^+_{p-k,(q+k+1,q+k+2)}>> v_{(p-k,q+k+2)} \\
@AA{e^-_{(p-k-1,p-k),q+k+1}}A @AA{e^-_{(p-k-1,p-k),q+k+2}}A \\
v_{(p-k-1,q+k+1)} @>e^+_{p-k-1,(q+k+1,q+k+2)}>> v_{(p-k-1,q+k+2)}. \\
\end{CD} 
\end{equation*}
Hence we have the following vertices and edges in Figure \ref{fig:fatzigzag}.
\begin{figure}[h]
\begin{equation*}
\begin{CD}
v_{(p,q)}@>{e^+_{p,(q,q+1)}}>>v_{(p,q+1)} @>{e^+_{p,(q+1,q+2)}}>>v_{(p,q+2)} 
@. @.
\\
@AA{e^-_{(p-1,p),q}}A @AA{e^-_{(p-1,p),q+1}}A @AA{e^-_{(p-1,p),q+2}}A @. 
@.\\
v_{(p-1,q)}@>{e^+_{p-1,(q,q+1)}}>> v_{(p-1,q+1)}@>{e^+_{p-1,(q+1,q+2)}}>>v_{(p-1,q+2)} 
@>{e^+_{p-1,(q+2,q+3)}}>> v_{(p-1,q+3)}@.
\\
@. @AA{e^-_{(p-2,p-1),q+1}}A @AA{e^-_{(p-2,p-1),q+2}}A 
@AA{e^-_{(p-2,p-1),q+3}}A 
@.\\
@. v_{(p-2,q+1)} @>{e^+_{p-2,(q+1,q+2)}}>> v_{(p-2,q+2)} 
    @>{e^+_{p-2,(q+2,q+3)}}>> v_{(p-2,q+3)} \cdots
\\
@. @. @AA{e^-_{(p-3,p-2),q+2}}A @AA{e^-_{(p-3,p-2),q+3}}A 
\\
@. @. v_{(p-3,q+2)} @>{e^+_{p-3,(q+2,q+3)}}>> v_{(p-3,q+3)}\cdots 
\\
@. @.  @. @AAA 
\\
@. @.  @. @. \cdots 
\end{CD}
\end{equation*}
\caption{A fat zigzag path}
\label{fig:fatzigzag}
\end{figure}By continuing these procedure, 
we finally obtain
a $(p,q)$-rectangle for which the upper left corner is $v_{(p,q)}$,
such as Figure \ref{fig:pqrec}.
\end{proof}
We note that by 
the left-resolving property of ${\mathcal{L}}^+$ and
the right-resolving property of ${\mathcal{L}}^-$,
a zigzag path $\gamma$ 
starting from a vertex $v_{(p,q)} \in V_{n(p,q)}$  
is determined by their vertices and their labeled sequences on the path.
This means that 
for a zigzag path $\gamma\in \CZ(v_{(p,q)})$ starting from $v_{(p,q)}$: 
$$
\gamma =
(e^+_{p-k,(q+k,q+k+1)},e^-_{(p-k-1,p-k),q+k+1})_{k=0}^\infty,
$$
by putting
\begin{align*}
v_{(p-k,q+k+1)} = & t(e^+_{p-k,(q+k,q+k+1)}) ( =t(e^-_{(p-k-1,p-k),q+k+1})), \\
v_{(p-k,q+k)} = & s(e^+_{p-k,(q+k,q+k+1)}) ( = s(e^-_{(p-k,p-k+1),q+k}))
\end{align*}
and
\begin{equation*}
x_{q+k} =  \lambda^+(e^+_{p-k,(q+k,q+k+1)}), \qquad
x_{p-k} =  \lambda^-(e^-_{(p-k-1,p-k),q+k+1}),
\end{equation*}
the path $\gamma$ is determined by their vertices 
$$v_{(p,q)}, v_{(p,q+1)},v_{(p-1,q+1)}, v_{(p-1,q+2)}, v_{(p-2,q+2)}, \dots  
$$
and the sequence of the labels   
$$
x_q, x_p, x_{q+1}, x_{p-1}, x_{q+2}, x_{p-2}, \dots 
$$
Hence the path $\gamma$ may be written
\begin{equation*}
\gamma =(v_{(p-k,q+k+1)}, v_{(p-k,q+k)}, x_{q+k}, x_{p-k})_{k\in \Zp},
\end{equation*}
where
$v_{(p-k,q+k+1)} \in V_{(p-k,q+k+1)},
v_{(p-k,q+k)} \in V_{(p-k,q+k)}.
$
In particular, for
$\gamma \in \CZ(v_{(-1,1)})$ with $p=-1, q=1$,
the path $\gamma$ is written such as 
\begin{equation}
\gamma =(v_{(-k-1,k+2)}, v_{(-k-1,k+1)}, x_{k+1}, x_{-k-1})_{k\in \Zp}.
\label{eq:gammavx}
\end{equation}
 As $v_{(-1,1)}\in V_{(-1,1)} =V_1$,
one may take a symbol $x_0 \in \Sigma$ 
together with $e^+_{-1,(0,1)}$ such that 
$\lambda^+(e^+_{-1,(0,1)}) = x_0$ 
and $t(e^+_{-1,(0,1)}) =v_{(-1,1)}$.
 Since the $\lambda$-graph bisystem $\LGBS$ satisfies FPCC, 
we have
 $P(v_{(-1,1)}) = F(v_{(-1,1)})
 $ 
so that 
$\lambda^-(e^-_{(-1,0),1}) = x_0$ and 
$s(e^-_{(-1,0),1}) =v_{(-1,1)}$
for some  $e^-_{(-1,0),1}$.
 We note that such directed edges 
 $e^+_{-1,(0,1)}, e^-_{(-1,0),1}$
 are uniquely determined by the symbol $x_0$ and the vertex $v_{(-1,1)}$
because of its left-resolving property of ${\mathcal{L}}^+$ and
its right-resolving property of ${\mathcal{L}}^-$.
  Let us denote by $\ZL$ the  set
\begin{equation*}
\ZL = \{
(x_0,\gamma) \in \Sigma\times \CZ_{(-1,1)} \mid
\exists 
e^+_{(0,1)} \in E^+_{0,1} ;
\lambda^+(e^+_{(0,1)}) = x_0, t(e^+_{(0,1)}) =v_{(-1,1)}
\}
\end{equation*}
of pairs of a zigzag path $\gamma$ starting from $v_{(-1,1)}$ and
a symbol $x_0$ preceding  $\gamma$
 that is figured as Figure \ref{fig:zigzag3}.
The symbol $x_0$ is called the head symbol of 
$(x_0,\gamma)\in \ZL$.
\begin{figure}[h]
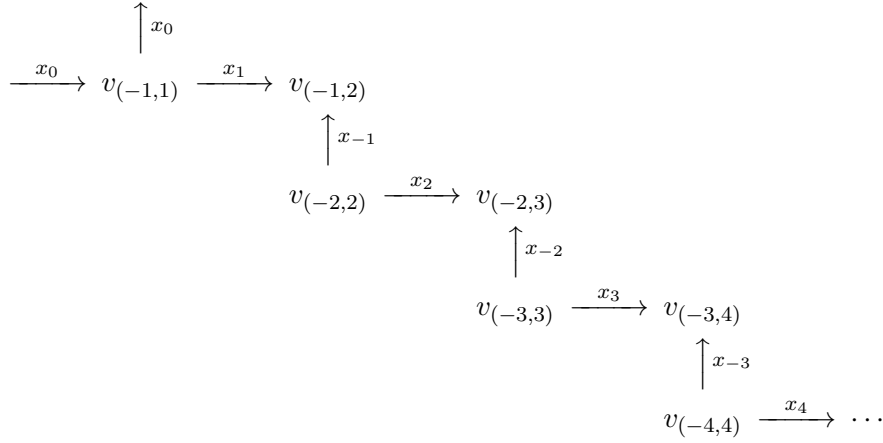

\begin{equation*}
\begin{CD}
@. @AA{x_0}A @. @. @.\\
@>{x_0}>>v_{(-1,1)}@>{x_1}>>v_{(-1,2)} @. @. @.
\\
@.@. @AA{x_{-1}}A @. @.  \\
@. @. v_{(-2,2)}@>{x_{2}}>>v_{(-2,3)} @. @. @. \\
@. @. @. @AA{x_{-2}}A @. @.@.\\
@. @. @. v_{(-3,3)} @>{x_{3}}>>v_{(-3,4)} @.@.\\
@. @. @. @. @AA{x_{-3}}A 
\\
@. @. @. @. v_{(-4,4)} @>{x_{4}}>>\cdots
\\
\\\end{CD}
\end{equation*}
\caption{A zigzag path with head symbol}
\label{fig:zigzag3}
\end{figure}
The label sequence  
\begin{equation*}
(\dots, x_{-4}, x_{-3}, x_{-2}, x_{-1}, x_0, x_1, x_2, x_3, x_4,\dots )
\end{equation*}
appearing in the above figure
is an element of the presented subshift $\Lambda_\L$ by the $\lambda$-graph bisystem $\LGBS$. 
We have the following proposition.
\begin{proposition}
Each element $(x_0,\gamma) \in \ZL$ uniquely extends to a whole configuration
on $\mbbZ2$ 
as an element of $\XL$.
Hence there exists a bijective correspondence between the 
set $\ZL$ of zigzag paths with head symbols 
and the set  $\XL$ of configurations on $\mbbZ2$.
\end{proposition} 
\begin{proof}
Let $(x_0,\gamma) \in \ZL$ be a zigzag path with head $x_0$.
We assume that $\gamma$ is given by the formula \eqref{eq:gammavx}.
By the preceding lemma, for the zigzag path
$\gamma \in \CZ(v_{(-1,1)})$ there uniquely exists a $(-1,1)$-rectangle 
$\square_{(-1,1)}$ for which upper left corner is $v_{(-1,1)}$.
Let $v_{(-1,0)}, v_{(0,1)} \in V_0$ 
be 
the source vertex of the edge in ${\mathcal{L}}^+$ labeled $x_0$,
the terminal vertex of the edge in ${\mathcal{L}}^-$ labeled $x_0$,
respectively.
We note that $v_{(-1,0)}= v_{(0,1)}$ because $V_0$ is singleton.
By the local property of $\lambda$-graph bisystem, there exist
$
v_{(-2,0)} \in V_{(-2,0)}
$
and
$
e^-_{(-2,-1),0}, \, e^+_{-2, (0,1)}
$
such that  
$
\lambda^-(e^-_{(-2,-1),0}) =x_{-1}, \, 
\lambda^+(e^+_{-2, (0,1)}) =x_{0}.
$
Similarly  there exist
$
v_{(0,2)} \in V_{(0,2)}
$
and
$
e^-_{(-1,0),2}, \, e^+_{0, (1,2)}
$
such that  
$
\lambda^-(e^-_{(-1,0),2}) =x_{0}, \, 
\lambda^+(e^+_{0, (1,2)}) =x_{1}.
$
The situation is figured such as in Figure \ref{fig:conflabel2}.
\begin{figure}
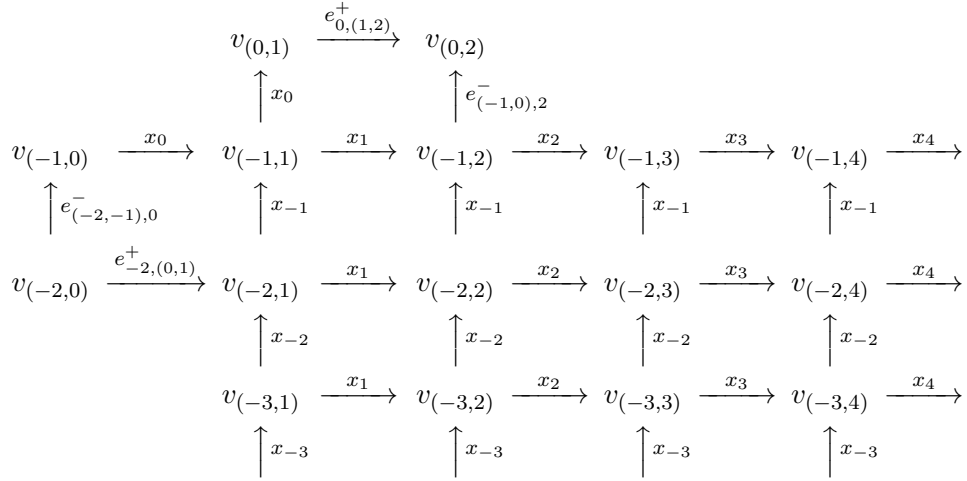

\begin{equation*}
\begin{CD}
 @. v_{(0,1)} @>{e^+_{0,(1,2)}}>>v_{(0,2)} @. @. @. \\
 @. @AA{x_0}A @AA{e^-_{(-1,0),2}}A    @. @. @. \\
 v_{(-1,0)}@>{x_0}>>v_{(-1,1)}@>{x_1}>>v_{(-1,2)} @>{x_2}>>v_{(-1,3)} @>{x_3}>>v_{(-1,4)}@>{x_4}>>\\
 @AA{e^-_{(-2,-1),0}}A @AA{x_{-1}}A @AA{x_{-1}}A @AA{x_{-1}}A @AA{x_{-1}}A @.\\
 v_{(-2,0)} @>{e^+_{-2,(0,1)}}>>v_{(-2,1)} @>{x_1}>>v_{(-2,2)} @>{x_2}>>v_{(-2,3)} 
@>{x_3}>>v_{(-2,4)}@>{x_4}>>\\
  @. @AA{x_{-2}}A @AA{x_{-2}}A @AA{x_{-2}}A @AA{x_{-2}}A @.\\
 @. v_{(-3,1)} @>{x_1}>>v_{(-3,2)} @>{x_2}>>v_{(-3,3)} @>{x_3}>>v_{(-3,4)}@>{x_4}>>\\
 @. @AA{x_{-3}}A @AA{x_{-3}}A @AA{x_{-3}}A @AA{x_{-3}}A @.\\
\end{CD}
\end{equation*}
\caption{Procedure from a zigzag path with head symbol}
\label{fig:conflabel2}
\end{figure}
We may continue these procedure to get vertices
$$
v_{(-k,0)} \in V_{(-k,0)}, \qquad v_{(0,k)} \in V_{(0,k)}, \qquad k=2,3, \dots
$$ 
and edges 
\begin{gather*}
e^-_{(-k-1,-k),0}, \qquad
e^+_{-k, (0,1)}, \qquad
e^-_{(-1,0),k}, \qquad
e^+_{0,(k,k+1)}, 
\end{gather*}
for $k=2,3,\dots $
such that 
\begin{gather*}
\lambda^-(e^-_{(-k-1,-k),0}) = x_{-k}, \quad
\lambda^+(e^+_{-k, (0,1)}) = x_0, \quad
\lambda^-(e^-_{(-1,0),k}) =x_0, \quad
\lambda^+(e^+_{0,(k,k+1)}) =x_k, 
\end{gather*}
for $k=2,3,\dots $
that are figured as in Figure \ref{fig:conflabel3}.
\begin{figure}
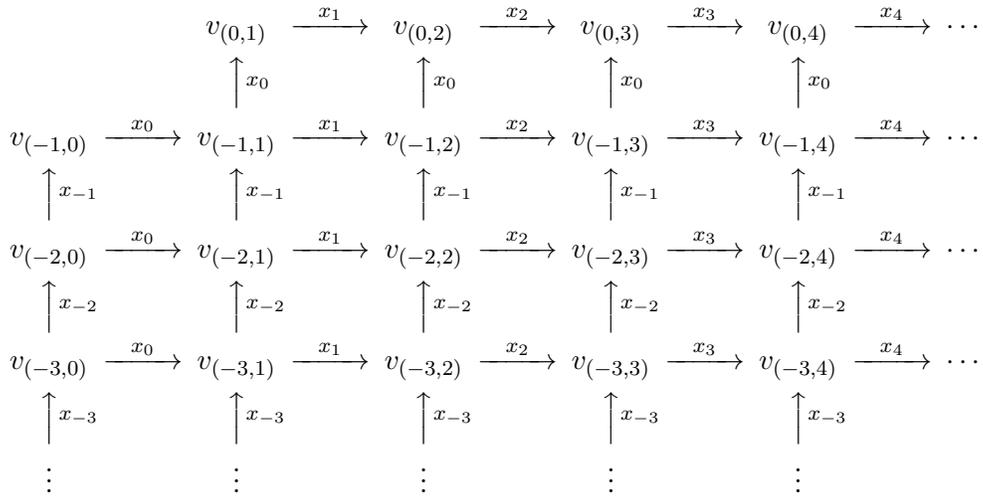

\begin{equation*}
\begin{CD}
 @. v_{(0,1)} @>{x_1}>>v_{(0,2)} @>{x_2}>>v_{(0,3)} @>{x_3}>>v_{(0,4)} @>{x_4}>> \cdots \\
 @. @AA{x_0}A @AA{x_0}A    @AA{x_0}A   @AA{x_0}A   @. \\
v_{(-1,0)}@>{x_0}>>v_{(-1,1)}@>{x_1}>>v_{(-1,2)} @>{x_2}>>v_{(-1,3)} @>{x_3}>>v_{(-1,4)}@>{x_4}>>\cdots 
\\
 @AA{x_{-1}}A @AA{x_{-1}}A @AA{x_{-1}}A @AA{x_{-1}}A @AA{x_{-1}}A @.\\
 v_{(-2,0)} @>{x_0}>>v_{(-2,1)} @>{x_1}>>v_{(-2,2)} @>{x_2}>>v_{(-2,3)} @>{x_3}>>v_{(-2,4)}@>{x_4}>>\cdots \\
   @AA{x_{-2}}A  @AA{x_{-2}}A @AA{x_{-2}}A @AA{x_{-2}}A @AA{x_{-2}}A @.\\
 v_{(-3,0)} @>{x_0}>> v_{(-3,1)} @>{x_1}>>v_{(-3,2)} @>{x_2}>>v_{(-3,3)} @>{x_3}>>v_{(-3,4)}@>{x_4}>>\cdots \\
 @AA{x_{-3}}A @AA{x_{-3}}A @AA{x_{-3}}A @AA{x_{-3}}A @AA{x_{-3}}A @.\\
\vdots @. \vdots @. \vdots @. \vdots @. \vdots @.  \\
\end{CD}
\end{equation*}
\caption{Further procedure from a zigzag path with head symbol}
\label{fig:conflabel3}
\end{figure}
\medskip

Since $x_{-1}\in P(v_{(-2,0)}) = F(v_{(-2,0)})$,
there exists $e^+_{-2,(-1,0)}$ such that 
$
\lambda^+(e^+_{-2,(-1,0)}) = x_{-1}, \,
t(e^+_{-2,(-1,0)}) = v_{(-2,0)}.
$
Put
$ v_{(-2,-1)} = s(e^+_{-2,(-1,0)}) \in V_0$.
By the local property of $\lambda$-graph bisystem and a similar manner to the previous way, 
we have vertices
$
v_{(-k,-1)} \in V_{(-k,-1)} $ and edges 
$
e^+_{-k,(-1,0)}, \, 
e^-_{(-k-1,-k), -1}
$
such that 
\begin{gather*}
v_{(-k,-1)} = s(e^+_{-k,(-1,0)} ) =s(e^-_{(-k,-k+1), -1} ), \\
t(e^+_{-k,(-1,0)}) = v_{(-k,0)}, \qquad 
s(e^+_{-k-1,(-1,0)} ) =s(e^-_{(-k-1,-k), -1} ), \\
\lambda^+(e^+_{-k,(-1,0)} ) = x_{-1}, \qquad
\lambda^-(e^-_{(-k-1,-k), -1} ) = x_{-k}
\end{gather*} 
for $k=2,3,\dots$
that is figured 
such as in Figure \ref{fig:conflabel5}. 
\begin{figure}
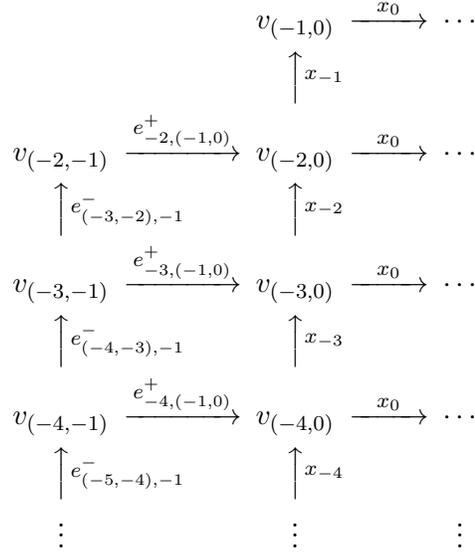

\begin{equation*}
\begin{CD}
 @. v_{(-1,0)} @>{x_0}>> \cdots \\
 @. @AA{x_{-1}}A @. \\
v_{(-2,-1)}@>{e^+_{-2, (-1,0)}} >>v_{(-2,0)}@>{x_0}>> \cdots \\
 @AA{e^-_{(-3,-2),-1}}A @AA{x_{-2}}A @. \\
 v_{(-3,-1)} @>{e^+_{-3,(-1,0)}}>>v_{(-3,0)} @>{x_0}>>\cdots \\
   @AA{e^-_{(-4,-3),-1}}A  @AA{x_{-3}}A \\
 v_{(-4,-1)} @>{e^+_{-4,(-1,0)}}>> v_{(-4,0)} @>{x_0}>>\cdots \\
 @AA{e^-_{(-5,-4),-1}}A @AA{x_{-4}}A @. \\
\vdots @. \vdots @. \vdots @.   \\
\end{CD}
\end{equation*}
\caption{Further procedure from a zigzag path}
\label{fig:conflabel5}
\end{figure}
\medskip

Therefore the element $(x_0,\gamma) \in \ZL$ extends to its lower left.
Similarly it can extend to its upper right, so that 
 $(x_0,\gamma) \in \ZL$ extends to a whole configuration on $\mbbZ2$ in a unique way.
\end{proof}

\begin{definition}
 A $\lambda$-graph bisystem $\LGBS$ is said to satisfy {\it condition}\/ (I)
if the cardinality 
$| \CZ(v_{(p,q)}) |$ of the set is more than or equal to two 
for any vertex $v_{(p,q)}\in V_{(p,q)}$ for every pair
$(p,q) \in \mbbZ2$.  
\end{definition} 
We then have the following proposition.
\begin{proposition}\label{prop:conditionI}
The configuration spase
$\XL$ is a Cantor set if and only if 
the $\lambda$-graph bisystem $\LGBS$ satisfies condition (I).
\end{proposition}
\begin{proof}
By the preceding discussion, we  know that $\XL$ is a totally disconnected
compact metric space.
We note that it is nonempty, 
because each member of $\XL$ bijectively corresponds to a member of
$\ZL$. 
By the definition of $\lambda$-graph bisystem, 
the set $\ZL$ is not empty.  
Hence 
$\XL$ is a Cantor set if and only if  it has no isolated points.

Suppose that  $\LGBS$ satisfies condition (I).
For $\mu \in P(v_{(p,q)})$ and the open neighborhood $U(\mu;v_{(p,q)})$ of $x$,
the set $\CZ(v_{(p,q)})$ contains more than or equal to two elements.
Hence the open neighborhood $U(\mu;v_{(p,q)})$ must contain
an element other than $x$. 
This shows that $\XL$ has no isolated points, proving it is a Cantor set.

Conversely, assume that $\XL$ is a Cantor set,
so that it has no isolated points.
Suppose that the cardinality of the set $\CZ(v_{(p,q)})$ is one
for some vertex $v_{(p,q)}$ for some $(p,q) \in \mbbZ2$,
and hence  $ \CZ(v_{(p,q)}) =\{\gamma \}$.
Take an admissible word $\mu \in P(v_{(p,q)})$ and its 
$(p,q)$-triangle determined by $\mu \in P(v_{(p,q)})$.
As $\gamma$ determines a $(p,q)$-rectangle in a unique way, 
the rectangle together with the $(p,q)$-triangle determines a unique whole configuration
on $\mbbZ2$,
that we denote by $x\in \XL$.
Then we see $x \in U(\mu;v_{(p,q)})$.
The open neighborhood however
does not contain any other point than $x$, because 
the cardinality of the set $\CZ(v_{(p,q)})$ is one. 
Therefore $x$ is an isolated point in $\XL$, a contradiction.
 \end{proof}


\section{Equivalence relation on the configuration space and its $C^*$-algebra}

We will define an equivalence relation on the configuration space
$\XL$, that will be shown to be AF and hence \'etale.
Two configurations $x, z \in \XL$  
 are said to be {\it equivalent}\/ if there exists $(p,q)\in \mbbZ2$
such that   
$\square_{(p,q)}(x) =\square_{(p,q)}(z)$ 
and written $x \sim z$.
If we specify $(p,q)$, then we write
$x \underset{(p,q)}{\sim} z$.
For $(p,q), (p',q') \in \mbbZ2$, 
define a partial order $\prec$  by 
$$
(p,q) \prec (p',q') \quad \text{ if } \quad  p' \le p \text{ and } q \le q'.
$$
The order 
$
(p,q) \prec (p',q')
$
means that $(p,q)$ locates at the upper left of $(p',q')$.
Hence
$$
x \underset{(p,q)}{\sim} z \quad \text{ implies }\quad  x \underset{(p',q')}{\sim} z
\quad \text{ if } \quad   (p,q) \prec (p',q') . 
$$
\begin{lemma}
The relation $\sim$ on $\XL$ is an equivalence relation.
\end{lemma}
\begin{proof}
Suppose that $x,y,z \in \XL$ satisfy
$x \sim y$ and $y \sim z$.
Take 
$(p,q), (r,s)\in \mbbZ2$
such that   
$\square_{(p,q)}(x) =\square_{(p,q)}(y)$ and 
$\square_{(r,s)}(y) =\square_{(r,s)}(z)$. 
Put $u_0 =\min\{p,r\}, v_0 =\max\{q,s\}$ 
so that 
$u_0 <v_0$.
We thus have
$\square_{(u_0, v_0)}(x) =\square_{(u_0,v_0)}(y) =\square_{(u_0,v_0)}(z)$ 
and hence
$x \sim z$. 
\end{proof}
We set the equivalence relation and its subequivalence relations:
\begin{gather*}
R_\L = \{
(x,z) \in {\XL} \times {\XL} \mid x \sim z\}, \\
R_{\L_{(p,q)}} = \{
(x,z) \in {\XL} \times {\XL} \mid x \underset{(p,q)}{\sim} z\} \quad \text{ for } (p,q) \in\mbbZ2.
\end{gather*}
Since we have 
\begin{equation*}
R_{\L_{(p,q)}} \subset R_{\L_{(p',q')}}
\quad \text{ if } \quad
(p,q) \prec (p',q'),
\end{equation*}
and
\begin{equation*}
\bigcup_{(p,q)\in \mbbZ2 }R_{\L_{(p,q)}} = R_\L,
\end{equation*}
the equivalence relation  $R_\L$
is an inductive limit of the subequivalence relations
$\{ R_{\L_{(p,q)}} \mid (p,q) \in \mbbZ2 \}.$
We endow $R_{\L_{(p,q)}}$ with the relative topology
of the product topology
$R_\L\times R_\L$ as a subset for each $(p,q) \in \mbbZ2$.
Then we endow 
$R_\L =\varinjlim R_{\L_{(p,q)}}$
with its inductive limit topology.
We identify $\X_\L$ with the subset
$\{(x,x) \in \X_\L\times\X_\L\}$
of $R_\L$ as a topological space. 
Let $s,r: R_\L\longrightarrow \XL \subset R_\L$ be two canonical projections
defined by
$$
s(x,z) = z, \qquad
r(x,z) = x \quad \text{ for } (x,z) \in R_\L.
$$
\begin{lemma}
The maps $s,r: R_\L\longrightarrow \XL \subset R_\L $  are local homeomorphisms,
so that the equivalence relation $R_\L$ is \'etale.
\end{lemma}
\begin{proof}
For any $(x,z)\in R_{\L_{(p,q)}}$ 
for some $(p,q) \in \mbbZ2$,
we have
$\square_{(p,q)}(x) = \square_{(p,q)}(z)$.
There exist $v_{(p,q)}\in V_{(p,q)}$
and words $\mu,\nu \in P(v_{(p,q)})$
such that 
$x \in U(\mu;v_{(p,q)}), z \in U(\nu;v_{(p,q)})$.
Let $V(\mu,\nu;v_{(p,q)})$ be an open neighborhood of $(x,z)$ in $R_\L$ 
defined by
$$
V(\mu,\nu;v_{(p,q)}) 
= R_{\L_{(p,q)}} \cap (U(\mu;v_{(p,q)}) \times U(\nu;v_{(p,q)})).
$$
Then we have open sets
$$
s(V(\mu,\nu;v_{(p,q)})) = U(\nu;v_{(p,q)}),
\qquad
r(V(\mu,\nu;v_{(p,q)})) = U(\mu;v_{(p,q)})
$$
in $\XL$.
By Lemma \ref{lem:muvx},
the maps
$
s: V(\mu,\nu;v_{(p,q)})\longrightarrow  U(\nu;v_{(p,q)}),
\,
r: V(\mu,\nu;v_{(p,q)}) \longrightarrow  U(\mu;v_{(p,q)})
$
are bijective. It is straightforward to see that they are homeomorphic.
Therefore $s,r: R_\L\longrightarrow \XL$  are local homeomorphisms.
\end{proof}
We thus have the following proposition.
\begin{proposition}
The equivalence relation $R_\L$ is an \'etale AF equivalence relation.
\end{proposition}
\begin{proof}
It suffices to show that 
$R_\L$ is an AF-equivalence relation.
For each $(p,q) \in \mbbZ2$, the subequivalence relation
$R_{\L_{(p,q)}}$ is closed in $\XL\times\XL$ and hence compact,
so that 
$R_{\L_{(p,q)}}$ is compact open in $R_\L$.
This shows that $R_\L$ is an AF-equivalence relation (cf. \cite[Theorem 6.17]{PutnamAMS}).
\end{proof}


We will next define irreducibility on $\L$.
\begin{definition}
A $\lambda$-graph bisystem $\LGBS$ is said to be {\it irreducible}\/
if for any $p \in \N$, 
a vertex $v_{(-p,p)} \in V_{(-p,p)}$
and a zigzag path $\gamma \in \CZ{(v_{(-1,1)})}$ starting from a vertex 
$v_{(-1,1)} \in V_1$,
there exists a zigzag path $\eta \in \CZ{(v_{(-p,p)})}$ starting from 
the vertex   $v_{(-p,p)}$ such that the path meets to $\gamma$,
that is, there exists a vertex $v_{(p', q')}\in V_{(p',q')}$ 
for some $p<p'<q'$ such that both $\gamma$ and $\eta$ pass the vertex $v_{(p', q')}$.
\end{definition}
In general, an equivalence relation $R$ on a compact metric space $X$
is said to be {\it minimal}\/ if  
for any point $x \in X$, its $R$-equivalence class
$[x]_R$ defined by 
$[x]_R =\{z \in X \mid x\sim z\}$
is dense in $X$ (cf. \cite[Definition 2.22]{PutnamAMS}).
With this definition, we have the following proposition.
\begin{proposition}\label{prop:irreducible}
A $\lambda$-graph bisystem $\LGBS$ is irreducible if and only if the equivalence relation
$R_\L$ on $\XL$ is minimal.
\end{proposition}
\begin{proof}
Suppose that a $\lambda$-graph bisystem $\LGBS$ is irreducible.
For any two configurations $x,z \in \XL$,
take a neighborhood $U(\mu,v_{(p,q)})$ of $z$,
where the vertex of $z$ at $(p,q)$ is $v_{(p,q)}$.
Let $\gamma \in \CZ{(v_{(-1,1)})}$ 
be the zigzag path obtained by restricting $x$.
By irreducibility,
there exists a zigzag path $\eta \in \CZ{(v_{(p,q)})}$ starting from the vertex
$v_{(p,q)}$ such that $\eta$ meets to $\gamma$ at a vertex $(p',q')$.
Consider the $(p',q')$-rectangle $\square_{(p',q')}(x)$ of $x$.
By using the zigzag path $\eta$ from $v_{(p,q)}$ to $v_{(p',q')}$,
the $(p',q')$-rectangle $\square_{(p',q')}(x)$ extends to $(p,q)$-rectangle.
As $\mu \in P(v_{(p,q)})$, there exists a unique configuration 
$z \in \XL$ obtained by an extension of the $(p,q)$-rectangle and the word $\mu$.
Since $\square_{(p',q')}(z) = \square_{(p',q')}(x)$, 
we have 
$z\in [x]_{R_\L}$ and 
$z \in U(\mu;v_{(p,q)})$.
This shows that 
$ [x]_{R_\L}$ is dense in $\XL$.

Conversely assume that $R_\L$ is minimal.
Suppose that $\LGBS$ is not irreducible,
so that there exist $p \in \N$,
a vertex $v_{(-p,p)}\in V_{(-p,p)}$
and a zigzag path $\gamma \in \CZ{(v_{(-1,1)})}$ starting from $v_{(-1,1)}$
such that 
any zigzag path $\eta \in \CZ{(v_{(-1,1)})}$
starting from the vertex $v_{(-p,p)}$
can not meet to $\gamma$.
Take $x_0\in \Sigma$ such that 
$(x_0,\gamma) \in \ZL$
and hence the pair
$(x_0,\gamma)$ uniquely extends to a configuration denoted by $x\in \XL$.
Take $\mu \in P(v_{(-p,p)})$.
Since any zigzag path $\eta \in \CZ{(v_{(-p,p)})}$ starting from $v_{(-p,p)}$
does not meet to $\gamma$, we have
$$
U(\mu;v_{(p,p)}) \cap [x]_{R_\L} = \emptyset.
$$
This shows that $R_\L$ is not minimal.
\end{proof}

\medskip

We will next study the $C^*$-algebra $C^*(R_\L)$ of the equivalence relation $R_\L$.  
Let us recall the construction of the $C^*$-algebra $C^*(R_\L)$ of the equivalence relation $R_\L$.
Denote by
$C_c(R_\L)$ the $*$-algebra of complex valued continuous compactly supported functions on $R_\L$
with its product and $*$-operation defined by 
\begin{gather*}
(f*g)(x,y) = \sum_{z\in [x]_{R_\L}} f(x,z)g(z,y), \\
f^*(x,y) =\overline{f(y,x)}, \qquad f,g \in C_c(R_\L), \quad (x,y) \in R_\L. 
\end{gather*}
We remark that the support $K$ of a function in $C_c(R_\L)$ is compact,
and the equivalence relation $R_\L$ is \'etale,
the set $K \cap \{(x,z) \in R_\L\mid z\in [x]_{R_\L}\}$ is finite
and moreover $\sup_{x\in \XL}| K \cap \{(x,z) \in R_\L\mid z\in [x]_{R_\L}\}| <\infty,$
 so that $f*g$ belongs to $C_c(R_\L)$.
  The $*$-algebra $C_c(R_\L)$ has a natural right $C(\XL)$-module structure with 
$C(\XL)$-valued inner product defined by
\begin{equation*}
(\xi\cdot f)(x,z) =\xi(x,z)f(z), \qquad
\langle\xi,\eta\rangle(x) 
=\sum_{z \in [x]_{R_\L}}\overline{\xi(x,z)}\eta(x,z)
\end{equation*}
for $\xi, \eta \in C_c(R_\L), f \in C(\XL).$
Denote by $\ell^2(R_\L)$ 
the Hilbert $C(\XL)$-module obtained by the completion
of $C_c(R_\L)$ by the norm induced by the inner product. 
Let
$ B(\ell^2(R_\L))$ denote the $C^*$-algebra of bounded adjointable right module maps
on $\ell^2(R_\L)$.
Define the left-regular representation
$\lambda: C_c(R_\L)\longrightarrow B(\ell^2(R_\L))$
by setting
$
\lambda(f)\xi = f* \xi
$
for
$
f, \xi \in C_c(R_\L).
$
The reduced $C^*$-algebra $C^*_{\red}(R_\L)$ of the equivalence relation
$R_\L$ is defined by the norm closure of the algebra
$\lambda(C_c(R_\L))$ on $B(\ell^2(R_\L)).$
We note that the full $C^*$-algebra $C^*_{\full}(R_\L)$ of the equivalence relation
$R_\L$ is defined by the norm completion  of the $*$-algebra
$C_c(R_\L)$ by the universal norm 
$
\| \, \cdot \,  \|_{univ}
$ 
defined by
$$
\| f  \|_{univ} =\sup\{
\|\pi(f)\| \mid \pi :C_c(R_\L)\longrightarrow B(H) *-\text{representation on a Hilbert space }H
\}.
$$
Now our equivalence relation $R_\L$ is \'etale and AF,
so that it is amenable. This implies that the two $C^*$-algebras
$C^*_{\red}(R_\L)$ and $C^*_{\full}(R_\L)$ are canonically isomorphic,
which we denote by $\FL$.
We know the  following result, that is direct from a general theory
of AF-algebras.
\begin{theorem}
The $C^*$-algebra $\FL$ is an AF-algebra. It is simple if and only if 
the $\lambda$-graph bisystem $\L$ is irreducible.  
\end{theorem}
\begin{proof}
Since the equivalence relation is AF, the algbera $\FL$ is AF.
It is well-known that 
the $C^*$-algebra of \'etale equivalence relation is simple if and only if the equivalence relation is minimal. 
By Proposition \ref{prop:irreducible}, we know that $C^*(R_\L)$ is simple if and only if $\L$ is irreducible. 
\end{proof}
There is a well-known classification result for general \'etale AF equivalence relations 
and its AF-algebras due to Bratteli--Elliott--Krieger
that we state below (cf. \cite[Theorem 9.2]{PutnamAMS}).
\begin{theorem}[{\cite[Bratteli]{Bratteli}, \cite[Elliott]{Elliott}, 
\cite[Krieger]{Kr80MathAnn}}]
For two $\lambda$-graph bisystems 
$\L_1=(\L_1^-,\L_1^+), \L_2=(\L_2^-,\L_2^+)$ satisfying FPCC.
The following are equivalent.
\hspace{7cm}
\begin{enumerate}
\renewcommand{\theenumi}{\roman{enumi}}
\renewcommand{\labelenumi}{\textup{(\theenumi)}}
\item The \'etale equivalence relations $R_{{\mathcal{L}_1}}$ and $R_{{\mathcal{L}_2}}$are isomorphic.
\item The AF-algebras $\F_{{\mathcal{L}_1}}$ and $\F_{{\mathcal{L}_2}}$ are isomorphic.
\item The K-groups $K_0(\F_{{\mathcal{L}_1}})$ and $K_0(\F_{{\mathcal{L}_2}})$
with the positions of the classes of  their units are order isomorphic.
\end{enumerate}
\end{theorem}
In our setting, we have one more structure on the equivalence relation
$R_\L$ on $\XL$.
That is the shift homeomorphism
$\sigma_\L$ on $\XL$ arising from the shift homeomorphism
on the presented  subshift $\Lambda_\L$.
The homeomorphism $\sigma_\L$ on $\XL$
is defined by setting
\begin{equation}
\sigma_\L(x) =\{ (v_{(k+1,l+1)}, e^-_{(k,k+1),l+1}, e^+_{k+1,(l+1,l+2)} ) \}_{(k,l) \in \mbbZ2}
\label{eq:sigmaL}
\end{equation}
for 
$
x =\{ (v_{(k,l)}, e^-_{(k-1,k),l}, e^+_{k,(l,l+1)}) \}_{(k,l) \in \mbbZ2} \in \XL. 
$
Since $v_{(k,l)} \in V_{(k,l)}$ and hence
$v_{(k+1,l+1)} \in V_{(k+1,l+1)}=V_{n(k+1,l+1)}=V_{n(k,l)} =V_{(k,l)},$
and similarly 
$e^-_{(k, k+1),l+1} \in E^-_{(k-1,k),l}, 
 e^+_{k+1,(l+1,l+2)} \in E^+_{k,(l,l+1)}
$ for 
 $e^-_{(k-1, k),l} \in E^-_{(k-1,k),l}, 
 e^+_{k,(l,l+1)} \in E^+_{k,(l,l+1)},
$
the above 
$\sigma_\L: \XL \longrightarrow \XL$
is well-defined.
It is straightforward to see that 
$\sigma_\L: \XL \longrightarrow 
\XL$ is actually a homeomorphism.
One knows that 
\begin{equation*}
x\sim z \quad \text{ if and only if } \quad 
\sigma_\L(x) \sim \sigma_\L(z) 
\end{equation*}
so that 
\begin{equation*}
\sigma_{\L}\times\sigma_{\L}(R_\L) = R_\L.
\end{equation*}
Hence
the homeomorphism $\sigma_\L\times\sigma_\L$
on $R_\L$ induces an automorphism on the $C^*$-algebra
$\FL$, that we denote by 
$\rho_\L$.
We are interested in not only 
the algebraic structure of the algebra $\FL$ 
but also the behavier of the automorphism
$\sigma_\L$ on $\FL$.

In  \cite{MaPre2019b}, we had introduced an AF-algbera written  
$\mathcal{F}_{\frak L}$ with the shift automorphism
$\sigma_{\frak L}$ on it
from a $\lambda$-graph bisystem $\L =(\L^-,\L^+)$
satisfying FPCC.
In the paper  \cite{MaPre2019b}, 
a $\lambda$-graph bisystem is denoted by
${\frak L} =({\frak L}^-,{\frak L}^+)$.
The   $\lambda$-graph bisystems $\L$ and ${\frak L}$ are 
the same objects.
By their construction of the both $C^*$-algebras
$\FL$ and $\mathcal{F}_{\frak L}$, we know they are actually isomorphic 
together with their shift automorphisms.
Namely, we have 
\begin{proposition}\label{prop:isomAF}
The $C^*$-algebra $\FL$ with shift automorphism
$\rho_\L$  is isomorphic to  the AF-algebra
$\mathcal{F}_{\frak L}$ together with the shift automorphism $\sigma_{\frak L}$
introduced in \cite{MaPre2019b}.
\end{proposition}


\section{The groupoid $\GL = R_\L\rtimes_{\sigma_{\L}}\Z$ and its $C^*$-algebra}
Recall that the shift homeomorphism $\sigma_\L$ on $\XL$ is defined in \eqref{eq:sigmaL}.
We define an \'etale groupoid $\GL$
from the equivalence relation $R_\L$ and the homeomorphism $\sigma_\L$ by setting
\begin{equation*}
\mathcal{G}_{\mathcal{L}} =\{
(x,n,z) \in \mathcal{X}_{\mathcal{L}}\times\Z\times\mathcal{X}_{\mathcal{L}}
\mid (\sigma_{\mathcal{L}}^n(x), z) \in R_{\mathcal{L}} \}. 
\end{equation*}
As the correspndence 
\begin{equation*}
(x,n,z) \in \GL \longrightarrow ((\sigma_{\mathcal{L}}^n(x), z), n) \in R_{\mathcal{L}}\times\Z
\end{equation*}
is bijective, a topology on $\GL$ is transfered from the product topology 
of $R_{\mathcal{L}}\times\Z$.
The unit space
$\mathcal{G}_{\mathcal{L}}^{(0)}$ of $\GL$ is defined by
\begin{equation*}
\GL^{(0)} =\{
(x,0,x) \in \mathcal{X}_{\mathcal{L}}\times\Z\times\mathcal{X}_{\mathcal{L}}
 \} 
\end{equation*}
that is identified with $\XL$ with their topology.
Define the maps
$s,r : \GL\longrightarrow \GL^{(0)}$
 by setting 
 $
 s(x,n,z) = (z,0,z), r(x,n,z) = (x,0,x).
$ 
It is routine to see that $\GL$ has a groupoid structure and the above topology on $\GL$
makes  $\GL$ a topological groupoid.
It is also routine to see that 
the groupoid $\GL$ is amenable and \'etale
because the equivalence relation $R_\mathcal{L}$ is so. 

We will define a certain essential freeness of the dynamical system
$(\XL, \sigma_{\mathcal{L}})$ that makes the groupoid
$\GL$ is essentially principal. 
An \'etale groupoid $\mathcal {G}$ is said to be essentially principal
if the interior $\Int(\mathcal{G}')$ of the isotropy bundle
$\mathcal{G}' =\{ g \in \mathcal{G} \mid s(g) = r(g) \}$ 
of $\mathcal{G}$ coincides with its unit space
$\G^{(0)}$ (\cite{Renault}, \cite{Renault2}, \cite{Renault3}).
We say that $(\XL, \sigma_{\mathcal{L}})$ is {\it equivalently essentially free}\/
if for each $n \in \Z$ with $n\ne0$,
the interior 
of the set $\{x \in \XL \mid (\sigma_{\mathcal{L}}^n(x), x) \in R_{\mathcal{L}} \}$
is empty:
$$
\Int(\{x \in \XL \mid (\sigma_{\mathcal{L}}^n(x), x) \in R_{\mathcal{L}} \}) =\emptyset.
$$
We then have
\begin{lemma}\label{lem:essfree}
The dynamical system
 $(\XL, \sigma_{\mathcal{L}})$ is equivalently essentially free
if and only if the groupoid
$\GL$ is essentially principal.
\end{lemma}
\begin{proof}
As we see
\begin{equation*}
\GL' = \bigcup_{n\in \Z} \{(x,n,x) \in \GL \mid (\sigma_{\mathcal{L}}^n(x), x) \in R_\L \},
\end{equation*}
we have 
\begin{equation*}
\Int(\GL') = \bigcup_{n\in \Z} 
\Int(\{(x,n,x) \in \XL\times\Z\times\XL \mid (\sigma_{\mathcal{L}}^n(x), x) \in R_\L \}).
\end{equation*}
Now 
\begin{equation*}
\Int(\{(x,0,x) \in \XL\times\Z\times\XL \mid (x, x) \in R_\L \}) =\XL,
\end{equation*}
so that we have
$\Int(\GL') =\XL$ if and only if 
\begin{equation*}
\Int(\{(x,n,x) \in \XL\times\Z\times\XL \mid (\sigma_{\mathcal{L}}^n(x), x) \in R_\L \})
= \emptyset \quad\text{ for all } n \in \Z, n\ne 0,
\end{equation*}
being equivalent to the condition that 
$(\XL, \sigma_{\mathcal{L}})$ is equivalently essentially free.
\end{proof}

\medskip


Let us review the construction of the $C^*$-algebra from an \'etale groupoid $\G$
(\cite{Renault}). 
It is a generalization of that of $C^*$-algebra from an \'etale equivalence relation.
Let $C_c(\G)$ be the algebra of complex valued compactly supported continuous functions on $\G$
that has a product structure of $*$-algebra defined by
\begin{gather*}
(f*g)(u) = \sum_{r(t) = r(u)} f(t)g(t^{-1}u), \\
f^*(u) =\overline{f(u^{-1})}, \qquad f,g \in C_c(\G), u \in \G. 
\end{gather*}
The $*$-algebra $C_c(\G)$ is a $C_0(\G^{(0)})$-right module with a 
$C_0(\G^{(0)})$-valued inner product given by
\begin{gather*}
(\xi\cdot f)(u) =\xi(u)f(s(u)), \qquad f \in C_0(\G^{(0)}), \\
\langle\xi,\eta\rangle(x) 
=\sum_{x = s(u)}\overline{\xi(u)}\eta(u), \qquad  \xi, \eta \in C_c(\G),  x \in \G^{(0)}.
\end{gather*}
Let us denote by $\ell^2(\G)$ 
the Hilbert $C^*$-right $C_0(\G^{(0)})$-module obtained by the completion
of $C_c(\G)$ by the norm induced by the inner product. 
Let 
$\lambda: C_c(\G)\longrightarrow B(\ell^2(\G))$
be the $*$-homomorphism defined by
$
\lambda(f)\xi = f* \xi
$
for
$
f, \xi \in C_c(\G).
$
The norm closure of $\lambda(C_c(\G))$ 
on $\ell^2(\G)$ is called the reduced groupoid $C^*$-algebra
of the \'etale groupoid
$\G$, that is denoted by  $C^*_{\red}(\G)$.
Similarly to the case of the construction of 
the  \'etale equivalence relation, 
we have the full groupoid $C^*$-algebra $C^*_{\full}(\G)$ of the \'etale groupoid $\G$
by the completion of the universal $C^*$-norm on $C_c(\G)$.
Now apply the construction of the groupoid $C^*$-algebra for our \'etale groupoid 
$\GL$ from a $\lambda$-graph bisystem $\L$.
As our groupoid $\GL$ is amenable, 
the two $C^*$-algebras 
$C^*_{\red}(\GL)$ and $C^*_{\full}(\GL)$ are canonically isomorphic,
that we denote by  $\RL$.
Since the groupoid $\GL$ is a semi-direct product 
$\GL = R_\L\rtimes_{\sigma_\L}\Z$
and the automorphism $\rho_\L$ on $\FL(=C^*(R_\L))$ is naturally defined by the homeomorphism
$\sigma_\L$ on $\XL$,
general theory of groupoid $C^*$-algebras tells us the following lemma. 
\begin{lemma}
The $C^*$-algebra $\R_{\mathcal{L}}$ 
is canonically isomorphic to the crossed product 
$\F_{\mathcal{L}}\rtimes_{\rho_{\mathcal{L}}}\Z.$
\end{lemma}
In the construction of the $C^*$-algebra
$C^*_{\red}(\GL)$, a continuous homomorphism
$f:\GL\longrightarrow \Z$ defines a one-parameter unitary group
$U_t(f), t \in \mathbb{T}$ on $\ell^2(\GL)$ by setting
\begin{equation*}
[U_t(f)\xi](x,n,z) = \exp(2\pi\sqrt{-1}f(x,n,z) t) \xi(x,n,z), \qquad
\xi \in \ell^2(\GL), \,\, (x,n,z) \in \GL.
\end{equation*}
It is routine to check that 
$\Ad(U_t(f))(g) \in C_c(\GL)$ for
$g \in C_c(\GL)$.
If in particular we define a groupoid homomorphism 
$d_{\mathcal{L}}: \mathcal{G}_\mathcal{L}\longrightarrow \Z$ 
by 
$$
d_{\mathcal{L}}(x,n,z)=n, \qquad (x,n,z) \in \GL.
$$  
Then the automorphism
$\Ad(U_t(d_\mathcal{L}))$
defined by $d_\mathcal{L}$ is denoted by $\gamma_{\L_t}.$
That is exactly corresponds to the dual action $\hat{\rho}_{\L_t}$
of the crossed product 
$\R_\L = \F_{\mathcal{L}}\rtimes_{\rho_\L}\Z$
by the shift automorphism $\rho_{\mathcal{L}}$ on $\FL$.
Now the groupoid  $\GL$ is \'etale, so that the unit space 
$\GL^{(0)}$ is open and closed in $\GL$. Hence 
$\XL$ is regarded as an open and closed subgroupoid of  $\GL$.
This shows that the commutative $C^*$-subalgebra 
$C(\XL)$ of complex valued continuous functions on $\XL$ is regarded 
as a $C^*$-subalgebra of $\RL$.
By Lemma \ref{lem:essfree}, we have the following lemma.  
\begin{lemma}[cf. {Renault \cite[Proposition 4.7]{Renault}}]
 The dynamical system
$(\XL, \sigma_{\mathcal{L}})$ is equivalently essentially free
if and only if 
the $C^*$-subalgebra $C(\XL)$ is maximal abelian in $\RL$.
\end{lemma}
Two \'etale equivalence relations $(R_i, X_i), i=1,2$ on comact metric spaces
$X_i, i=1,2 $ are said to be isomorphic if there exists a homeomorphism
$h:X_1\longrightarrow X_2$ such that
$h\times h(R_1) = R_2$ and $h\times h: R_1\longrightarrow R_2$
is a homeomorphism in the topology of the equivalence relations $R_i, i=1,2$
(\cite[Definition 6.10]{PutnamAMS}). 
\begin{lemma}\label{lem:fgPhi}
Let $\mathcal{L}_i, i=1,2$
be $\lambda$-graph bisystems satisfying FPCC. 
Let $f:\G_{\mathcal{L}_1}\longrightarrow \Z$ and 
$g:\G_{\mathcal{L}_2}\longrightarrow \Z$ be continuous groupoid homomorphisms.
Suppose that there exists an isomorphism
$\varphi:\G_{\mathcal{L}_1}\longrightarrow \G_{\mathcal{L}_2}$ of \'etale groupoids
such that  $f = g\circ \varphi$.
Then there exists an isomorphism 
$\Phi:\R_{\mathcal{L}_1}\longrightarrow \R_{\mathcal{L}_2}$ 
of $C^*$-algebras such that 
$$
\Phi(C(\X_{\mathcal{L}_1}))= C(\X_{\L_2})\quad\text{ and }
\quad \Phi\circ\Ad(U_t(f)) = \Ad(U_t(g))\circ\Phi, \quad t\in \T.
$$ 
\end{lemma}
\begin{proof}
Let us denote by 
$s_i : \G_{\L_i}\longrightarrow \G_{\L_i}^{(0)},\,
  r_i : \G_{\L_i}\longrightarrow \G_{\L_i}^{(0)},i=1,2
$
the source maps and the range maps of the groupoids $\G_{\L_i}, i=1,2$.
For the isomorphism 
$\varphi:\G_{\L_1}\longrightarrow \G_{\L_2}$
of \'etale groupoids,
let us denote by 
$h=\varphi|_{\G_{\L_1}^{(0)}}: \G_{\L_1}^{(0)} \longrightarrow {\G_{\L_2}^{(0)}}$
the restriction of $\varphi$ to its unit space.
It is regarded as a homeomorphism
$h:{\X_{\L_1}}\longrightarrow {\X_{\L_2}}$
from ${\X_{\L_1}}$ to $\X_{\L_2}$ such that 
$s_2 \circ\varphi = h\circ s_1$ and
$r_2 \circ\varphi = h\circ r_1.$
Define unitaries
$V_h: \ell^2(\G_{\L_2})\longrightarrow \ell^2(\G_{\L_1})$ and
$V_{h^{-1}}: \ell^2(\G_{\L_1})\longrightarrow \ell^2(\G_{\L_2})$
by setting
\begin{align}
[V_h\zeta](x,n,z) 
& = \zeta(\varphi(x,n,z)), \qquad \zeta\in \ell^2(\G_{\L_2}), \, (x,n,z) \in \G_{\L_1}, 
\label{eq:Vh}\\
[V_{h^{-1}}\xi](y,m,w) 
& = \xi(\varphi^{-1}(y,m,w)), \qquad \xi\in \ell^2(\G_{\L_1}), \, (y,m,w) \in \G_{\L_2}
\label{eq:Vinvh}
\end{align}
and we set 
$\Phi:= \Ad(V_{h^{-1}})$.
As in the proof of \cite[Proposition 5.6]{MaCJM}, we have 
$\Phi(C_c(\G_{\L_1})) =  C_c(\G_{\L_2})$ and
$\Phi(C(\X_{\L_1})) =  C(\X_{\L_2})$.
Hence
$\Phi(\R_{\L_1}) = \R_{\L_2}.$
It is straightforward to see that the equalities 
for $a \in C_c(\G_{\L_2}), \zeta \in \ell^2(\G_{\L_2})$ and $(y,m,w) \in \G_{\L_2}$
\begin{align*}
& [(\Phi\circ\Ad(U_t(f)))(a)\zeta](y,m,w) \\
= & \sum_{\gamma\in \G_{\L_1}, r(\gamma) = h^{-1}(y)} 
a(\gamma) \exp(-2\pi\sqrt{-1}f(\gamma^{-1}) t) 
      \zeta(\varphi(\gamma^{-1})\cdot (y,m,w)), \\
& [(\Ad(U_t(g)) \circ\Phi)(a)\zeta](y,m,w) \\
= & \sum_{
\gamma\in \G_{\L_1}, r(\gamma) = h^{-1}(y)} 
a(\gamma) \exp(-2\pi\sqrt{-1}g(\varphi(\gamma^{-1})) t) 
      \zeta(\varphi(\gamma^{-1})\cdot (y,m,w))
\end{align*}
hold. As
$f(\gamma^{-1}) = g(\varphi(\gamma^{-1}))$, we have 
$\Phi\circ\Ad(U_t(f))(a) = \Ad(U_t(g))\circ\Phi(a)$
and hence $\Phi\circ\Ad(U_t(f)) =\Ad(U_t(g))\circ\Phi$.
\end{proof}
We have the following proposition.
\begin{proposition}
Let $\mathcal{L}_i, i=1,2$
be $\lambda$-graph bisystems
satisfying FPCC. 
Assume that
the dynamical systems
$(\X_{\L_i}, \sigma_{\L_i}), i=1,2$ are equivalently essentially free.
Then the following assertions are equivalent.
\begin{enumerate}
\renewcommand{\theenumi}{\roman{enumi}}
\renewcommand{\labelenumi}{\textup{(\theenumi)}}
\item The \'etale groupoids  $\G_{{\mathcal{L}_1}}$ and $\G_{{\mathcal{L}_2}}$ are isomorphic.
\item 
There exists an isomorphism 
$\Phi:\R_{\mathcal{L}_1}\longrightarrow \R_{\mathcal{L}_2}$ 
of $C^*$-algebras such that 
$
\Phi(C(\X_{\mathcal{L}_1}))= C(\X_{\L_2})
$ 
and 
$$
\Phi\circ\gamma_{{\L_1}_t} = \Ad(U_t(c_{\L_2}))\circ\Phi,\qquad
\Phi\circ\Ad(U_t(c_{\L_1})) = \gamma_{{\L_2}_t} \circ\Phi,
 \quad t\in \T
$$ 
for some continuous homomorphisms
$c_{\L_1}: \G_{\L_1}\longrightarrow \Z$ and
$c_{\L_2}: \G_{\L_2}\longrightarrow \Z$.
\end{enumerate}
\end{proposition}
\begin{proof}
(i) $\Longrightarrow$ (ii):
Suppose that 
there exists an isomorphism
$\varphi:\G_{{\mathcal{L}_1}} \longrightarrow \G_{{\mathcal{L}_2}}$ 
of \'etale groupoids.
Let 
$h:=\varphi|_{\G_{{\mathcal{L}_1}}^{(0)}}: 
\G_{{\mathcal{L}_1}}^{(0)}\longrightarrow\G_{{\mathcal{L}_2}}^{(0)}
$ 
be the restriction of $\varphi$ to its unit space $\G_{{\mathcal{L}_1}}^{(0)}$.
Hence there exist continuous groupoid homomorphisms
$c_{\L_1}:\G_{{\mathcal{L}_1}} \longrightarrow \Z$
and
$c_{\L_2}:\G_{{\mathcal{L}_2}} \longrightarrow \Z$
such that 
\begin{gather*}
\varphi(x,n,z) = (h(x), c_{\L_1}(x,n,z), h(z)), \qquad (x,n,z) \in \G_{\L_1}, \\
\varphi^{-1}(y,m,w) = (h^{-1}(y), c_{\L_2}(y,m,w), h^{-1}(w)), \qquad (y,m,w) \in \G_{\L_2}
\end{gather*}
so that 
$d_{\L_2}\circ \varphi = c_{\L_1}$ and $d_{\L_1}\circ \varphi^{-1} = c_{\L_2},
$
and hence $c_{\L_2}\circ \varphi = d_{\L_1}$.
By Lemma \ref{lem:fgPhi}, we obtain
$
\Phi\circ\gamma_{\L_1,t} = \Ad(U_t(c_{\L_2}))\circ\Phi
$ and
$
\Phi\circ\Ad(U_t(c_{\L_1})) = \gamma_{\L_2,t} \circ\Phi,
$ proving (ii).

(ii) $\Longrightarrow$ (i):
Since the \'etale groupoids $\G_{{\mathcal{L}_1}}$ and $\G_{{\mathcal{L}_2}}$
are both essentially principal,
Renault's result \cite[Proposition 4.11]{Renault} tells us that 
an isomorphism
$\Phi: C^*_{\red}(\G_{\L_1}) (=\R_{\mathcal{L}_1})
\longrightarrow C^*_{\red}(\G_{\L_2}) (=\R_{\mathcal{L}_2})$ 
of $C^*$-algebras satisfying 
$
\Phi(C(\G_{\L_1}^{(0)}))= C(\G_{\L_2}^{(0)})
$ 
(that is, 
$
\Phi(C(\X_{\mathcal{L}_1}))= C(\X_{\L_2})
$) 
yields an isomorphism
of the underlying \'etale groupoids
between $\G_{{\mathcal{L}_1}}$ and $\G_{{\mathcal{L}_2}}$.
\end{proof}

The AF-algebra $\FL$ is canonically isomorphic to the fixed point algebra 
$(\RL)^{\gamma_\L}$ of $\RL$ under the dual action $\gamma_\L$.
We regard the AF-algebra $\FL$ as the subalgebra $(\RL)^{\gamma_\L}$ of $\RL$. 
The following lemma is also due to J. Renault's result \cite[Proposition 30]{Renault}. 
\begin{lemma}\label{lem:groupoidC*6.2}
Assume that
the dynamical systems
$(\X_{\L_i}, \sigma_{\L_i}), i=1,2$ are equivalently essentially free.
Then the following assertions are equivalent.
\begin{enumerate}
\renewcommand{\theenumi}{\roman{enumi}}
\renewcommand{\labelenumi}{\textup{(\theenumi)}}
\item
There exists an isomorphism
$\varphi: \G_{\L_1} \longrightarrow \G_{\L_2}$  
of \'etale groupoids such that 
$\varphi(R_{\L_1})= R_{\L_2}$  
and
$\varphi(\X_{\L_1})= \X_{\L_2}.$  
\item
There exists an isomorphism 
$\Phi:\R_{\mathcal{L}_1}\longrightarrow \R_{\mathcal{L}_2}$ 
of $C^*$-algebras satisfying 
$\Phi(\F_{\L_1}) = \F_{\L_2}$ and
$\Phi(C(\X_{\L_1}))= C(\X_{\L_2}).$ 
 \end{enumerate}
\end{lemma}
\begin{proof}
Since 
the dynamical systems
$(\X_{\L_i}, \sigma_{\L_i}), i=1,2$ are equivalently essentially free,
the \'etale groupoids 
$\G_{\L_1}$ and $\G_{\L_2}$
are both essentially principal by Lemma \ref{lem:essfree}.
The implication
(i) $\Longrightarrow$ (ii) 
is direct.

By Renault \cite[Proposition 4.11]{Renault2},
an isomorphism 
$\R_{\L_1} \longrightarrow \R_{\L_2}$
of $C^*$-algebras such that 
$\Phi(C(\X_{\L_1})) = C(\X_{\L_2})$
yields an isomorphism $\varphi$ between the underlying
\'etale groupoids
$\G_{\L_1}$ and $\G_{\L_2}$.
Now the AF-algebras $\F_{\L_i}$ is isomorphic to
the groupoid algebra $C^*(R_{\L_i}), i=1,2.$
Its construction of the isomorphism $\varphi$
of the \'etale groupoids
tells us  that
$
\varphi(R_{\L_1})=R_{\L_2}
$  
by the additional condition
$\Phi(\F_{\L_1}) = \F_{\L_2}$,
thus proving the implication (ii) $\Longrightarrow$ (i).
\end{proof}


\begin{proposition}\label{prop:main6.1}
Assume that
the dynamical systems
$(\X_{\L_i}, \sigma_{\L_i}), i=1,2$ are equivalently essentially free.
Suppose that 
there exists an isomorphism 
$\Phi: \R_{\mathcal{L}_1} \longrightarrow \R_{\mathcal{L}_2}
$ of $C^*$-algebras such that
$\Phi(C({\mathcal{X}}_{\mathcal{L}_1})) =C({\mathcal{X}}_{\mathcal{L}_2})$
and
$\Phi\circ\gamma_{{\L_1}_t}=\gamma_{{\L_2}_t}\circ\Phi$,
$t \in \T$.
Then there
exists an isomorphism 
$\varphi:\mathcal{G}_{\mathcal{L}_1}\longrightarrow  \mathcal{G}_{\mathcal{L}_2}$
of the \'etale groupoids such that
$d_{\mathcal{L}_2} =\varphi\circ d_{\mathcal{L}_1}$.
\end{proposition}
\begin{proof}
Suppose that 
there exists an isomorphism 
$\Phi: \R_{\mathcal{L}_1} \longrightarrow \R_{\mathcal{L}_2}$
of $C^*$-algebras such that 
$\Phi(C({\mathcal{X}}_{\mathcal{L}_1})) =C({\mathcal{X}}_{\mathcal{L}_2})$
and
$\Phi\circ\gamma_{{\L_1}_t}=\gamma_{{\L_2}_t}\circ\Phi$,
$t \in \T$.
Since 
the fixed point algebra
$(\R_{\L_i})^{\gamma_{\L_i}}$ of
 ${\R_{\mathcal{L}_i}}$ under $\gamma_{\L_i}$ 
is canonically isomorphic to the AF algebra
$\F_{\L_i}$,
 the isomorphism 
$\Phi: \R_{\mathcal{L}_1} \longrightarrow \R_{\mathcal{L}_2}$
satisfies
$\Phi(\F_{\L_1}) = \F_{\L_2}$.
By Lemma \ref{lem:groupoidC*6.2},
we then find  
an isomorphism
$\varphi: \G_{\L_1} \longrightarrow \G_{\L_2}$
of \'etale groupoids  and 
a homeomorphism 
$h: \X_{\L_1} \longrightarrow \X_{\L_2}$
such that  
$\varphi(R_{\L_1}) =R_{\L_2}, \, 
\varphi|_{\X_{\L_1}} = h$
and
$\Phi(f) = f\circ h^{-1}$ for $f \in C(\X_{\L_1})$.
For the isomorphism
 $\varphi:\G_{\L_1} \longrightarrow \G_{\L_2}$ 
of \'etale groupoids,
there exists a continuous groupoid homomorphism
$c_{\L_1}:\G_{{\mathcal{L}_1}} \longrightarrow \Z$
and
$c_{\L_2}:\G_{{\mathcal{L}_2}} \longrightarrow \Z$
such that 
\begin{gather*}
\varphi(x,n,z) = (h(x), c_{\L_1}(x,n,z), h(z)), \qquad (x,n,z) \in \G_{\L_1}, \\
\varphi^{-1}(y,m,w) = (h^{-1}(y), c_{\L_2}(y,m,w), h^{-1}(w)), \qquad (y,m,w) \in \G_{\L_2}.
\end{gather*}
We put
$c_1^n(x) = c_{\L_1}(x,n,\sigma_{\L_1}^n(x))$ for $ x \in \X_{\L_1}, n \in \Z$
and
$d_1(x, z) = c_{\L_1}(x,0, z)$ for $(x,z ) \in R_{\L_1}$.
By the identities
\begin{gather*}
(x,n,z) = (x,n, \sigma_{\L_1}^n(x))\cdot (\sigma_{\L_1}^n(x), 0,z),
\qquad (x,n,z) \in \G_{\L_1}, \\
(x,n+m, \sigma_{\L_1}^{n+m}(x)) 
= (x,n,\sigma_{\L_1}^n(x))\cdot (\sigma_{\L_1}^n(x), m, \sigma_{\L_1}^{n+m}(x)),
\qquad x\in \X_{\L_1}, \, n,m\in\Z,
\end{gather*}
we have 
\begin{gather} 
c_{\L_1}(x,n,z) = c_1^n(x) + d_1(\sigma_{\L_1}^n(x),z), \qquad (x,n,z) \in \G_{\L_1}, \label{eq:cl1}\\
c_1^{n+m}(x) = c_1^n(x) + c_1^m(\sigma_{\L_1}^n(x)), \qquad x \in \X_{\L_1},
\, n,m\in\Z. \label{eq:c1nm}
\end{gather}
For $(x,z) \in R_{\L_1}$, we have
\begin{equation*}
\varphi(x,0,z)
= (h(x),c_{\L_1}(x,0,z),h(z))
= (h(x), d_1(x,z),h(z)).
\end{equation*}
As 
$\varphi(x,0,z)\in R_{\L_2}$,
one knows that $d_1(x,z)=0$
for $(x,z) \in R_1$.
To prove $d_{\L_2} \circ\varphi =d_{\L_1}$,
it suffices to show the equality
$c_1^n(x) = n$ for all $x \in \X_{\L_1}$ and $n \in \mathbb{Z}$.
We put $c_1(x)= c_1^1(x), x \in \X_{\L_1}$,  and will prove that
$c_1(x) =1$ for all $x \in \X_{\L_1}$.   

Let
$V_h$ and $V_{h^{-1}}$ be the unitaries defined in \eqref{eq:Vh} and \eqref{eq:Vinvh}.
As in the proof of Lemma \ref{lem:fgPhi},
by putting
$\Phi_h = \Ad(V_{h^{-1}})$,
we have
an isomorphism
$\Phi_h: \R_{\L_1} \longrightarrow \R_{\L_2}$
of $C^*$-algebras such that 
$\Phi_h(C({\mathcal{X}}_{\mathcal{L}_1})) =C({\mathcal{X}}_{\mathcal{L}_2})$,
more exactly $\Phi_h(f) = f \circ h^{-1}$ for $f \in C(\X_{\L_1})$,
and
$$
\Phi_h\circ \gamma_{{\L_1}_t} = \Ad(U_t(c_{\L_2}))\circ \Phi_h,
\qquad
\Phi_h\circ\Ad(U_t(c_{\L_1})) =\gamma_{{\L_2}_t}\circ \Phi_h,
\qquad t \in \mathbb{T}.
$$
Let $U_{\L_1}$ be the unitary on $\ell^2(\G_{\L_1})$ defined by 
\begin{equation*}
(U_{\L_1}\xi)(x,n,z) = \xi(\sigma_{\L_1}(x), n-1, z), \qquad \xi \in \ell^2(\G_{\L_1}), \, (x,n,z) \in \G_{\L_1}, 
\end{equation*}
that is regarded as the implementing unitary of the positive generator of 
the group representation of $\Z$ in the crossed product
$C^*(R_{\L_1})\rtimes_{\rho_{\L_1}}\Z$.
We then have the 
equality
$U_{\L_1} f U_{\L_1}^* = f\circ \sigma_{\L_1}^{-1}$ for $f \in C(\X_{\L_1})$.
Since  $\Phi(f) = \Phi_h(f)$ for all $f \in C(\X_{\L_1})$,
we see that 
\begin{equation*}
\Phi(U_{\L_1}) \Phi_h(f) \Phi(U_{\L_1}^*) 
= \Phi( f\circ \sigma_{\L_1}^{-1})= \Phi_h( f\circ \sigma_{\L_1}^{-1})
= \Phi_h(U_{\L_1} f  U_{\L_1}^*),
\end{equation*}
so that 
$
\Phi_h^{-1}(\Phi(U_{\L_1})) f \Phi_h^{-1}(\Phi(U_{\L_1}^*)) 
= U_{\L_1} f U_{\L_1}^*
$
and hence
\begin{equation*}
U_{\L_1}^*\Phi_h^{-1}(\Phi(U_{\L_1})) f  
=  f U_{\L_1}^*\Phi_h^{-1}(\Phi(U_{\L_1}))
\quad
\text{ for all }
f \in C(\X_{\L_1}).
\end{equation*}
Since  
the dynamical system
$(\X_{\L_1}, \sigma_{\L_1})$ is equivalently essentially free,
the groupoid 
$\G_{\L_1}$ is essentially principal 
by Lemma \ref{lem:essfree}. 
By \cite[Proposition 4.7]{Renault} (cf. \cite[Proposition 4.2]{Renault2}), 
$C(\X_{\L_1})=C(\G_{\L_1}^{(0)})$
is a maximal abelian $C^*$-subalgebra of $\R_{\L_1}.$
Hence there exists a unitary $u_1 \in C(\X_{\L_1})$ 
satisfying 
$U_{\L_1}^*\Phi_h^{-1}(\Phi(U_{\L_1})) = u_1$,
so that  we have 
\begin{equation}
\Phi(U_{\L_1}) = \Phi_h(U_{\L_1} u_1).\label{eq:PhiUphi}
\end{equation}
Since
$\Phi\circ\gamma_{{\L_1}_t}=\gamma_{{\L_2}_t} \circ \Phi $
and
$
\Phi_h\circ\Ad(U_t(c_{\L_1})) =\gamma_{{\L_2}_t}\circ \Phi_h,
$
the equality \eqref{eq:PhiUphi} tells us
\begin{equation}
\Phi\circ \gamma_{{\L_1}_t}(U_{\L_1}) 
= \Phi_h\circ\Ad(U_t(c_{\L_1}))(U_{\L_1} u_1),
\label{eq:PhiUphi2}
\end{equation}
so that the equality \eqref{eq:PhiUphi2} implies
\begin{equation}
\exp{(2\pi\sqrt{-1}t)}\Phi(U_{\L_1}) 
= \Phi_h(U_t(c_{\L_1}) U_{\L_1} u_1 U_t(c_{\L_1})^*)
\label{eq:PhiUphi3}
\end{equation}
because $\gamma_{{\L_1}_t}(U_{\L_1}) = \exp(2\pi\sqrt{-1}t)U_{\L_1}$.
Since 
we have  
\begin{align*}
[U_t(-c_{\L_1})u_1\xi](x,n,z)
& =\exp{(2\pi\sqrt{-1} (-c_{\L_1}(x,n,z))t)}[u_1\xi](x,n,z) \\
& =u_1(x) \exp{(2\pi\sqrt{-1} (-c_{\L_1}(x,n,z))t)}\xi(x,n,z) \\
& =[u_1(x) U_t(-c_{\L_1})\xi](x,n,z) 
\end{align*}
for $\xi \in l^2(\G_{\L_1})$ and $(x,n,z) \in \G_{\L_1}$,
 we see
$U_t(c_{\L_1})^*u_1 =u_1 U_t(c_{\L_1})^*$.
Hence the equality \eqref{eq:PhiUphi3} 
goes to 
\begin{equation*}
\exp{(2\pi\sqrt{-1}t)}\Phi(U_{\L_1}) 
= \Phi_h(U_t(c_{\L_1}) U_{\L_1}  U_t(c_{\L_1})^*u_1)
= \Phi_h(U_t(c_{\L_1}) U_{\L_1}  U_t(c_{\L_1})^*)
\end{equation*}
because of \eqref{eq:PhiUphi}, so that we have 
\begin{equation}
\exp{(2\pi\sqrt{-1}t)}U_{\L_1} 
= U_t(c_{\L_1}) U_{\L_1}  U_t(c_{\L_1})^*.
\label{eq:PhiUphi6}
\end{equation}
Since we easily know that 
\begin{align*}
  & [U_t(c_{\L_1}) U_{\L_1}  U_t(c_{\L_1})^*\xi](x,n,z) \\
= & \exp{(2\pi\sqrt{-1} ( c_{\L_1}(x,n,z) -c_{\L_1}(\sigma_{\L_1}(x), n-1,z))t)}\xi(\sigma_{\L_1}(x),n-1,z) 
\end{align*}
and
\begin{equation*}
[\exp{(2\pi\sqrt{-1}t)}U_{\L_1}\xi](x,n,z) 
= \exp{(2\pi\sqrt{-1}t)}\xi(\sigma_{\L_1}(x),n-1,z)
\end{equation*}
for  $\xi \in l^2(\G_{\L_1})$ and 
$(x,n,z) \in \G_{\L_1}$, 
the equality \eqref{eq:PhiUphi6}
ensures us the identity
\begin{equation*}
c_{\L_1}(x,n,z) -c_{\L_1}(\sigma_{\L_1}(x), n-1,z) =1,
\qquad (x,n,z) \in \G_{\L_1}.
\end{equation*}
By the formulas \eqref{eq:cl1} and \eqref{eq:c1nm}, the equalities 
\begin{align*}
  &c_{\L_1}(x,n,z) -c_{\L_1}(\sigma_{\L_1}(x), n-1,z) \\
=&\{c_1^n(x) + d_1(\sigma_{\L_1}^n(x),z)\} 
  -\{c_1^{n-1}(\sigma_{\L_1}(x)) + d_1(\sigma_{\L_1}^{n-1}(\sigma_{\L_1}(x)), z) \}
= c_1(x)
\end{align*}
hold, proving
$c_1(x) =1$ for all $x \in \X_{\L_1}$.
\end{proof}


\medskip

We thus have the following theorem.
\begin{theorem}\label{thm:GRmain}
Suppose that $(\X_{\mathcal{L}_i},\sigma_{\mathcal{L}_i}), i=1,2$ are equivalently essentially free.
The following two assertions are equivalent.
\hspace{7cm}
\begin{enumerate}
\renewcommand{\theenumi}{\roman{enumi}}
\renewcommand{\labelenumi}{\textup{(\theenumi)}}
\item
There exists an isomorphism 
$\varphi:\mathcal{G}_{\mathcal{L}_1}\longrightarrow  \mathcal{G}_{\mathcal{L}_2}$
of the \'etale groupoids such that
$d_{\mathcal{L}_2} =\varphi\circ d_{\mathcal{L}_1}$.
\item
There exists an isomorphism 
$\Phi: \R_{\mathcal{L}_1} \longrightarrow \R_{\mathcal{L}_2}
$ of $C^*$-algebras such that
$\Phi(C({\mathcal{X}}_{\mathcal{L}_1})) =C({\mathcal{X}}_{\mathcal{L}_2})$
and
$\Phi\circ\gamma_{{\L_1}_t}=\gamma_{{\L_2}_t}\circ\Phi$,
$t \in \T$.
\end{enumerate}
\end{theorem}
\begin{proof}
(i) $\Longrightarrow$ (ii):
Assume that there exists an isomorphism 
$\varphi:\mathcal{G}_{\mathcal{L}_1}\longrightarrow  \mathcal{G}_{\mathcal{L}_2}$
of the \'etale groupoids such that
$d_{\mathcal{L}_2} =\varphi\circ d_{\mathcal{L}_1}$.
By Lemma \ref{lem:fgPhi} with the equality 
$d_{\mathcal{L}_2} =\varphi\circ d_{\mathcal{L}_1}$,  
 there exists an isomorphism 
$\Phi:\R_{\mathcal{L}_1}\longrightarrow \R_{\mathcal{L}_2}$ 
of $C^*$-algebras such that 
$$
\Phi(C(\X_{\mathcal{L}_1}))= C(\X_{\L_2})\quad\text{ and }
\quad \Phi\circ\Ad(U_t(d_{\L_1})) = \Ad(U_t(d_{\L_2}))\circ\Phi, \quad t\in \T
$$ 
so that $\Phi\circ\gamma_{{\L_1}_t}=\gamma_{{\L_2}_t}\circ\Phi$,
$t \in \T$.

The assertion (ii) $\Longrightarrow$ (i)
comes from Proposition \ref{prop:main6.1}.
\end{proof}


\section{Application to subshifts}
Let $\L = (\L^-,\L^+)$ be a $\lambda$-graph bisystem over $\Sigma$ satisfying FPCC.
Recall that its presenting subshift $\Lambda_\L$ 
 is defined by the subshift over $\Sigma$ whose admissible words 
consist of the set of words
$\cup_{v \in V} P(v)$,
where $V = \cup_{l=0}^\infty V_l$ 
is the common vertex set of $\L$.
For 
$
x =\{ (v_{(k,l)}, e^-_{(k-1,k),l}, e^+_{k,(l,l+1)})\}_{(k,l) \in \mbbZ2} \in \XL,
$
put
\begin{equation}
x_k =\lambda^-(e^-_{(k-1,k), l}) \qquad \text{ for } (k,l) \in \mbbZ2, \label{eq:xk}
\end{equation}
that does not depend on $l$.
It is also written as
$
x_l =\lambda^+(e^+_{k, (l,l+1)}) 
$
for
$ (k,l) \in \mbbZ2.
$
The set of the sequences
$(x_n)_{n\in \Z}$ defines the subshift
$\Lambda_{\L}$.

Take and fix a real number 
$r_\circ$ such as 
$0<r_\circ <1$.
The shift space $\Lambda_{\L}$ is a compact metric space  by the metric defined by for 
$a =(a_n)_{n\in \Z}, b=(b_n)_{n\in \Z}$ with $a \ne b$
\begin{equation*}
d(a,b) 
=  
\begin{cases}
1 & \text{ if  }  a_0 \ne b_0,\\
(r_\circ)^{m} & \text{ if } m=\Max\{ n \mid a_k = b_k \text{ for all }k \text{ with }
|k| < n\}.
\end{cases}
\end{equation*}
Let us denote by  $\sigma_{\Lambda_\L}: \Lambda_\L \longrightarrow \Lambda_\L$
the shift homeomorphism defined by
$\sigma_{\Lambda_\L}((a_n)_{n\in \Z}) =(a_{n+1})_{n\in \Z}$ for $(a_n)_{n\in \Z} \in \Lambda_\L$.
We consider the asymptotic equivalence relation
$G^a_{\Lambda_\L}$ on $\Lambda_\L$ by setting
\begin{equation*}
G^a_{\Lambda_\L} =\{
(a,b) \in \Lambda_\L\times\Lambda_\L \mid
\lim_{n\to\infty}d(\sigma_{\Lambda_\L}^n(a),\sigma_{\Lambda_\L}^n(b))
 =
\lim_{n\to\infty}d(\sigma_{\Lambda_\L}^{-n}(a),\sigma_{\Lambda_\L}^{-n}(b)) =0\}.
\end{equation*}
If the subshift $\Lambda_\L$ is a topological Markov shift,
the topological dynamical system $(\Lambda_\L, \sigma_{\Lambda_\L})$
becomes a Smale space (\cite{Putnam1}, \cite{Ruelle1}).
For Smale spaces the equivalence relation $G_{\Lambda_\L}^a$ 
was called an asymptotic equivalence, that was first studied in Ruelle \cite{Ruelle1}
and Putnam \cite{Putnam1}.  

Define a map
 $\pi: \XL\longrightarrow \Lambda_{\L}$
by setting
\begin{equation}\label{eq:pi}
\pi(x)= (x_n)_{n\in \Z}\quad \text{ for }
x =\{ (v_{(k,l)}, e^-_{(k-1,k),l}, e^+_{k,(l,l+1)})\}_{(k,l) \in \XL},
\end{equation}
where $x_n$ is defined by \eqref{eq:xk}.
It is easy to see that 
 $\pi: \XL\longrightarrow \Lambda_{\L}$
is a continuous surjection, called a factor map. 
Then we have the following lemma,
that is easy to prove. 
\begin{lemma} \hspace{7cm} 
\begin{enumerate}
\renewcommand{\theenumi}{\roman{enumi}}
\renewcommand{\labelenumi}{\textup{(\theenumi)}}
\item
$\pi\circ \sigma_{\L} = \sigma_{\Lambda_\L}\circ\pi$.
\item 
$\pi\times\pi(R_\L) \subset G^a_{\Lambda_\L}$.
\end{enumerate}
 \end{lemma}

Let $(\Lambda, \sigma_\Lambda)$ be a subshift over $\Sigma$.
As in Section 2, we have the canonical $\lambda$-graph bisystem written 
$\L_\Lambda = (\L_\Lambda^-,\L_\Lambda^+)$ from $\Lambda$.
The presenting subshift $\Lambda_{\L_{\Lambda}}$
coincides with the original subshift 
$\Lambda$. We then know the following proposition, that is direct from the construction of
 $\L_\Lambda = (\L_\Lambda^-,\L_\Lambda^+)$ from $\Lambda$.

\begin{proposition}
For a subshift $(\Lambda,\sigma_\Lambda)$, 
let $\pi:\X_{\Lambda_\L} \longrightarrow \Lambda$ 
be the factor map defined by \eqref{eq:pi}. 
\begin{enumerate}
\renewcommand{\theenumi}{\roman{enumi}}
\renewcommand{\labelenumi}{\textup{(\theenumi)}}
\item
$\Lambda$ is a topological Markov shift if and only if $\pi$ is injective and hence an isomorphism.
\item
$\Lambda$ is a sofic shift if and only if $\pi$ is a finite to one factor map.
\end{enumerate}
\end{proposition}
For a subshift $\Lambda$ and its canonical $\lambda$-graph bisystem 
$\L_\Lambda = (\L_\Lambda^-,\L_\Lambda^+),$
we have the associated configuration space
$\X_{\L_{\Lambda}}$ with shift homeomorphism $\sigma_{\L_\Lambda}$,
the \'etale equivalence relation $R_{\L_{\Lambda}}$,
the \'etale groupoid  $\G_{\L_{\Lambda}}$,
the $C^*$-algebras $\F_{\L_{\Lambda}}$,
$\R_{\L_{\Lambda}}$ and its dual action $\gamma_{\L_{\Lambda}}$.
They are denoted by
$\X_{\Lambda}$ with $\sigma_{\X_{\Lambda}}$,
 $R_{\Lambda}$, $\G_{\Lambda}$,
$\F_{\Lambda}$,
$\R_{\Lambda}$ and $\gamma_\Lambda$,
respectively.
The groupoid homomorphism
$d_{\L_\Lambda} :  \G_{\L_{\Lambda}}\longrightarrow \Z$
 is also denoted by
$d_\Lambda$.

\begin{proposition}\label{prop:topconjugate1}
Suppose that two-sided subshifts $\Lambda_1$ and $\Lambda_2$ are topologically conjugate
via a homeomorphism $h_\Lambda: \Lambda_1\longrightarrow \Lambda_2$.
Then there exists a homeomorphism
$h: \mathcal{X}_{\Lambda_1}\longrightarrow \mathcal{X}_{\Lambda_2}$
satisfying the following conditions:
\begin{enumerate}
\renewcommand{\theenumi}{\roman{enumi}}
\renewcommand{\labelenumi}{\textup{(\theenumi)}}
\item 
$h_{\Lambda}\circ \pi_1 = \pi_2\circ h$,
where $\pi_i: \X_{\Lambda_i}\longrightarrow \Lambda_i, i=1,2$
is the factor map defined by \eqref{eq:pi}.
\item 
$h\circ \sigma_{\X_{\Lambda_1}} =\sigma_{\X_{\Lambda_2}}\circ h$.
\item
$h\times h(R_{\Lambda_1}) = R_{\Lambda_2}$.
\item The map
$h\times h :R_{\Lambda_1} \longrightarrow R_{\Lambda_2}$
is a homeomorphism with respect to the topology on the \'etale equivalence relations
$R_{\Lambda_1}$ and $R_{\Lambda_2}$.
\end{enumerate}
\end{proposition}
\begin{proof}
Suppose that subshifts $\Lambda_1$ over alphabet $\Sigma_1$  
and $\Lambda_2$ over alphabet $\Sigma_2$  are topologically conjugate.
As in the paper of Nasu \cite{Nasu}, \cite{NasuMemoir},
they are related by a finite chain of bipartitely related subshifts.
Hence we may assume that  
there exist finite alhabets $C, D$ and injections
$\phi_1: \Sigma_1\longrightarrow C\cdot D$,
$\phi_2: \Sigma_2\longrightarrow D\cdot C$
and a bipartite subshift
$(\widehat{\Lambda}, \hat{\sigma})$ over alphabet $C\cup D$,
such that 
both subshifts 
$\Lambda_1$ and $\Lambda_2$ are regarded as 
disjoint subsystems of its $2$-higher block shift $\widehat{\Lambda}^{(2)}$
through the maps $\phi_1$ and $\phi_2$, respectively.
Then the restriction of the shift $\hat{\sigma}$ to 
the subshift $\Lambda_1$ 
gives rise to a topological conjugacy from 
$\Lambda_1$ to $\Lambda_2$.
Hence  $(k,l)$-centrally equivalent pair
$x\overset{c}{\underset{(k,l)}{\sim}}z$ in $\Lambda_1$
goes to
$\hat{\sigma}(x)\overset{c}{\underset{(k,l)}{\sim}}\hat{\sigma}(z)$ in $\Lambda_2$.
Therefore we may identify $(k,l)$-centrally equivalence classes in
$\Lambda_1$ with those in $\Lambda_2$ by shifting on $\widehat{\Lambda}$,
so that
their configuration spaces 
$\mathcal{X}_{\Lambda_1}$ and  
$\mathcal{X}_{\Lambda_2}$ together with
their equivalence relations
$R_{\Lambda_1}$ and $R_{\Lambda_2}$
are also identified
by taking shift $\hat{\sigma}$.
Since $\hat{\sigma}^2$ on $\widehat{\Lambda}$
corresponds to their shifts on their original
subshifts $\Lambda_1$ and $\Lambda_2$,
we have the assertion (iv) above.
\end{proof} 

We note that the statements (iii) and (iv) of Proposition \ref{prop:topconjugate1}
imply that the equivalence relations 
$R_{\Lambda_1}$ and $R_{\Lambda_2}$ are isomorphic
(see \cite[Definition 6.10]{PutnamAMS}) if the subshifts 
$\Lambda_1$ and $\Lambda_2$ are topologically conjugate.

\begin{proposition}\label{prop:topconjugate2}
The following assertions are equivalent: 
\begin{enumerate}
\renewcommand{\theenumi}{\roman{enumi}}
\renewcommand{\labelenumi}{\textup{(\theenumi)}}
\item Two-sided subshifts $\Lambda_1$ and $\Lambda_2$ are topologically conjugate. 
\item There exist an isomorphism
$\varphi: \G_{\Lambda_1}\longrightarrow \G_{\Lambda_2}$
of \'etale groupoids 
and a homeomorphism 
$h_\Lambda:{\Lambda_1}\longrightarrow {\Lambda_2}$
of the shift spaces such that 
\begin{equation*}
\varphi(R_{\Lambda_1}) = R_{\Lambda_2},
\quad
h_\Lambda\circ\pi_1 = \pi_2\circ h,
\quad
d_{\Lambda_2}\circ \varphi = d_{\Lambda_1}
\end{equation*}
where $h =\varphi|_{\G_{\Lambda_1}^{(0)}}: 
\G_{\Lambda_1}^{(0)}\longrightarrow \G_{\Lambda_2}^{(0)}$
is the restriction of $\varphi$ to its unit space, and 
$\pi_i :\X_{\Lambda_i}\longrightarrow \Lambda_i, i=1,2$
is the factor map defined by \eqref{eq:pi}. 
\end{enumerate}
\end{proposition}
\begin{proof}
(i) $\Longrightarrow$ (ii):
Suppose that two subshifts  $\Lambda_1$ and $\Lambda_2$ are topologically conjugate. 
By Proposition \ref{prop:topconjugate1}, 
there exists  a homeomorphim
$h: \mathcal{X}_{\Lambda_1}\longrightarrow \mathcal{X}_{\Lambda_2}$
satisfying the conditions (i), (ii), (iii) and (iv) in Proposition \ref{prop:topconjugate1}.
Define $\varphi(x,n,z) = (h(x), n, h(z))$ for $(x,n,z) \in \G_{\L_1}$. 
Since
$(\sigma_{\X_{\Lambda_1}}^n(x), z) \in R_{\Lambda_1}$,
we have 
$(h(\sigma_{\X_{\Lambda_1}}^n(x)), h(z)) \in R_{\Lambda_2}$,
because of  Proposition \ref{prop:topconjugate1} (iii).
As 
$h\circ\sigma_{\X_{\Lambda_1}} =\sigma_{\X_{\Lambda_2}}\circ h$,
we have 
$(\sigma_{\X_{\Lambda_2}}^n(h(x)), h(z)) \in R_{\Lambda_2}$
so that 
$\varphi(x,n,z) = (h(x), n, h(z))$ belongs to $R_{\Lambda_2}$.
Hence we may define
$\varphi: \G_{\Lambda_1}\longrightarrow \G_{\Lambda_2}$.
It is easy to see that 
$\varphi: \G_{\Lambda_1}\longrightarrow \G_{\Lambda_2}$
is an isomorphism of \'etale groupoids
satisfying the desired properties. 

(ii) $\Longrightarrow$ (i):
We identify the configuration spaces $\X_{\Lambda_i}$ with
the unit spaces $\G_{\Lambda_i}^{(0)}, i=1,2$, respectively.
For $x \in \X_{\Lambda_1}$ and $n \in \Z$,
we have 
$(x, n, \sigma_{\X_{\Lambda_1}}^n(x)) \in \G_{\Lambda_1}$,
and 
$\varphi(x, n, \sigma_{\X_{\Lambda_1}}^n(x)) \in \G_{\Lambda_2}$.
By the condition
$
d_{\Lambda_2}\circ \varphi = d_{\Lambda_1},
$
we have 
$$
\varphi(x, n, \sigma_{\X_{\Lambda_1}}^n(x)) =(h(x), n,h(\sigma_{\X_{\Lambda_1}}^n(x))).  
$$
In particular, for $n=1$,
we have
$(h(x), 1,h(\sigma_{\X_{\Lambda_1}}(x))) \in \G_{\Lambda_2}$
so that  
$$
(\sigma_{\X_{\Lambda_2}}(h(x)), h(\sigma_{\X_{\Lambda_1}}(x))) \in R_{\Lambda_2}.
$$
Since $\varphi:R_{\Lambda_1}\longrightarrow R_{\Lambda_2}$
is a homeomorphism with respect to their \'etale topology,
there exist continuous maps $p,q : \X_{\Lambda_1}\longrightarrow \Z$
such that $(p(x), q(x)) \in \mbbZ2$
and 
$$
\square_{(p(x),q(x))}(\sigma_{\X_{\Lambda_2}}(h(x)))
=\square_{(p(x),q(x))}(h(\sigma_{\X_{\Lambda_1}}(x))),
$$
so that 
\begin{equation}\label{eq:pi2sigma}
\pi_2(\sigma_{\X_{\Lambda_2}}(h(x)))_n
=\pi_2(h(\sigma_{\X_{\Lambda_1}}(x)))_n
\quad
\text{ for } n \le p(x) \, \text{ or }\, q(x) \le n.
\end{equation}
Since 
\begin{equation}\label{eq:pi2circh}
\pi_2\circ h = h_\Lambda\circ\pi_1, \qquad
\pi_i\circ\sigma_{\X_{\Lambda_i}} =\sigma_{\Lambda_i}\circ\pi_i, \quad, i=1,2,
\end{equation}
by \eqref{eq:pi2sigma}, \eqref{eq:pi2circh}, we have
\begin{equation}\label{eq:sigmaLambda2x}
(\sigma_{\Lambda_2}\circ h_\Lambda(\pi_1(x)))_n
=(h_\Lambda\circ\sigma_{\Lambda_1}(\pi_1(x)))_n
\quad
\text{ for } n \le p(x) \, \text{ or }\, q(x) \le n.
\end{equation}
Put
$p_1= \min\{p(x) \mid x \in \X_{\Lambda_1}\}$ and
$q_1= \max\{q(x) \mid x \in \X_{\Lambda_1}\}$
so that $(p_1,q_1) \in \mbbZ2$.
As $\pi_1:\X_{\Lambda_1}\longrightarrow \Lambda_1$ 
is surjective,
we have
\begin{equation}\label{eq:sigmaLambda2a}
(\sigma_{\Lambda_2}\circ h_\Lambda(a))_n
=(h_\Lambda\circ\sigma_{\Lambda_1}(a))_n
\quad
\text{ for } a\in \Lambda_1, \quad n \le p_1\, \text{ or }\, q_1 \le n.
\end{equation}
Similarly 
there exist $p_2, q_2\in \mbbZ2$ such that 
\begin{equation}\label{eq:sigmaLambda1b}
(\sigma_{\Lambda_1}\circ h_\Lambda^{-1}(b))_n
=(h_\Lambda^{-1}\circ\sigma_{\Lambda_2}(b))_n
\quad
\text{ for } b\in \Lambda_2, \quad n \le p_2 \, \text{ or }\, q_2 \le n.
\end{equation}
Put
$K := \max\{ |p_1|, |q_1|, |p_2|, |q_2| \}$ so that we have 
\begin{align}
\sigma_{\Lambda_2}^{K+1}(h_\Lambda(a))_n 
&= \sigma_{\Lambda_2}^K(h_{\Lambda_1}(a))_n, 
\quad
\text{ for } a\in \Lambda_1, \quad  n\ge 0, \label{eq:sigma2K1}\\
\sigma_{\Lambda_1}^{K+1}(h_\Lambda^{-1}(b))_n 
&= \sigma_{\Lambda_1}^K(h_{\Lambda_2}^{-1}(b))_n, 
\quad
\text{ for } b\in \Lambda_2, \quad n\ge 0. \label{eq:sigma1K1}
\end{align}
The above two equalities
\eqref{eq:sigma2K1} and \eqref{eq:sigma1K1}
are the same equalities (3.1) and (3.3) in \cite[Definition 3.1]{Ma2019JOT}.
By \cite[Theorem 3.3]{Ma2019JOT} and its proof,
noticing the remark before \cite[Theorem 3.3]{Ma2019JOT},
we know that two topological Markov shifts
 $\Lambda_1$ and $\Lambda_2$ are topologically conjugate.
\end{proof}
For a subshift $\Lambda$, the factor map
$\pi: \X_\Lambda \longrightarrow \Lambda$ gives rise to a natural inclusion
$C(\Lambda) \subset C(\X_\Lambda)$ of commutative $C^*$-algebras.
We henceforth regard $C(\Lambda)$ as a $C^*$-subalgebra of $C(\X_\Lambda)$.   
\begin{proposition}\label{prop:topconjugate3}
Suppose that the dynamical systems 
$(\X_{\Lambda_i}, \sigma_{\X_{\Lambda_i}}), i=1,2$ are equivalently essentially free.   
The following assertions are equivalent: 
\begin{enumerate}
\renewcommand{\theenumi}{\roman{enumi}}
\renewcommand{\labelenumi}{\textup{(\theenumi)}}
\item There exist an isomorphism
$\varphi: \G_{\Lambda_1}\longrightarrow \G_{\Lambda_2}$
of \'etale groupoids 
and a homeomorphism
$h_\Lambda: \Lambda_1\longrightarrow \Lambda_2$ 
such that 
$$
\varphi(R_{\Lambda_1}) = R_{\Lambda_2},
\quad
h_\Lambda\circ\pi_1 = \pi_2\circ h,
\quad
d_{\Lambda_2}\circ \varphi = d_{\Lambda_1}.
$$
\item
There exists an isomorphism of the $C^*$-algebra 
$\Phi:  \R_{\Lambda_1} \longrightarrow \R_{\Lambda_2}$ 
such that 
$$
\Phi(C(\X_{\Lambda_1})) = C(\X_{\Lambda_2}), \qquad
\Phi(C(\Lambda_1)) = C(\Lambda_2), \qquad
\Phi\circ \gamma_{{\Lambda_1}_t} 
= \gamma_{{\Lambda_2}_t}\circ\Phi, \quad t \in \mathbb{T}.
$$
\end{enumerate}
\end{proposition}
\begin{proof}
Add the extra conditions $h_\Lambda\circ\pi_1 = \pi_2\circ h$ and 
$\Phi(C(\Lambda_1)) = C(\Lambda_2)$
to the statements (i) and (ii) in Theorem \ref{thm:GRmain}, respectively. 
By following the proof of Theorem \ref{thm:GRmain},
we may prove the assertion.
\end{proof}
Therefore we have the following theorem
by Proposition \ref{prop:topconjugate2} and 
    Proposition \ref{prop:topconjugate3}.
\begin{theorem}\label{thm:GRmain2}
Suppose that the dynamical systems 
$(\X_{\Lambda_i}, \sigma_{\X_{\Lambda_i}}), i=1,2$ are equivalently essentially free.   
Two-sided subshifts $\Lambda_1, \Lambda_2$ are topologically conjugate
if and only if  there exists an isomorphism  
$\Phi:  \R_{\Lambda_1} \longrightarrow \R_{\Lambda_2}$
of $C^*$-algebras such that 
$$
\Phi(C(\X_{\Lambda_1})) = C(\X_{\Lambda_2}), \qquad
\Phi(C(\Lambda_1)) = C(\Lambda_2), \qquad
\Phi\circ \gamma_{{\Lambda_1}_t} 
= \gamma_{{\Lambda_2}_t}\circ\Phi, \quad t \in \mathbb{T}.
$$
Hence the quadruplet  
$(\R_\Lambda, C(\X_\Lambda), C(\Lambda), \gamma_{\Lambda})$ is a complete invariant of topological conjugacy of subshift $\Lambda$ satisfying the condition 
 that $(\X_\L,\sigma_{\X_\L})$ is equivalently essentially free.
\end{theorem}

\medskip
Let $\LGBS$ be a $\lambda$-graph bisystem satisfying FPCC
and $\Lambda_\L$ be its presenting subshift.
Recall that 
$\GL'$ denotes its isotropy bundle of the groupoid $\GL$, that is, 
$$
\GL' = \{ \gamma \in \GL \mid s(\gamma) = r(\gamma) \}
=\cup_{n\in \Z}\{(x,n,x) \in \XL\times\Z\times\XL \mid  (\sigma_\L^n(x),x) \in R_\L\}.
$$
\begin{definition}
A subshift $\Lambda_\L$ is said to satisfy $\pi$-{\it condition}\/ (I)
if for any $\gamma \in \GL$ with $\gamma \not\in \GL'$
and
open neighborhood $U \subset \GL$ of $\gamma$,
there exists $\gamma_1 \in U$  such that 
$$
\pi(s(\gamma_1)) \ne \pi(r(\gamma_1)),$$
where
$\pi:\XL\longrightarrow \Lambda_\L$ is the factor map
defined by \eqref{eq:pi}.
\end{definition}
\begin{lemma}
Assume that the subshift $\Lambda_\L$ satisfies $\pi$-condition (I).
If $f \in \RL$ commutes with $C(\Lambda_\L)$, then 
$f$ vanishes outside $\Int(\GL')$.
\end{lemma}
\begin{proof}
Suppose that
$f \in C_c(\GL)$ commutes with $C(\Lambda_\L)$.
Identifying $\XL$ with $\GL^{(0)}$, we have
for $a \in C(\XL)$ and $\gamma\in \GL$,
$$
(f*a)(\gamma) = f(\gamma) a(s(\gamma)), \qquad
(a*f)(\gamma) = a(r(\gamma))f(\gamma),
$$
so that 
$f*a = a*f $ if and only if 
$f(\gamma) (a(s(\gamma))-a(r(\gamma))) =0$
for all $\gamma \in \GL.$
In particular for $a = b\circ \pi$ for some $b \in C(\Lambda_\L)$,
we have
$f*b = b*f $ if and only if 
$f(\gamma) (b(\pi(s(\gamma)))-b(\pi(r(\gamma))) )=0$
for all $\gamma \in \GL.$ 
By the assumption that 
$f \in C_c(\GL)$ commutes with $C(\Lambda_\L)$,
we have 
\begin{equation}
f(\gamma) (b(\pi(s(\gamma))-b(\pi(r(\gamma))) =0\quad
\text{ for all } \quad b \in C(\Lambda_\L),\,  \gamma \in \GL. \label{eq:fb1}
\end{equation}
Suppose that $f$ does not vanish outside  $\Int(\GL')$.
There exist $\gamma_0 \in \GL\backslash \Int(\GL')$
and an open neighborhood $U\subset \GL$ of $\gamma_0$
such that $f(\gamma) \ne 0$ for all $\gamma\in U$. 
As $\gamma_0 \not\in \Int(\GL')$, one may find $\gamma' \in U$
such that $\gamma' \not\in\GL'$.
By the  hypothesis that $\Lambda$ satisfies $\pi$-condition (I),
there exists $\gamma_1 \in U$  such that 
$$
\pi(s(\gamma_1)) \ne \pi(r(\gamma_1)).
$$
One may take $b_1 \in C(\Lambda_\L)$ such that 
$b_1(\pi(s(\gamma_1))) =1, 
b_1(\pi(r(\gamma_1))) =0$,
a contradiction to \eqref{eq:fb1}.
\end{proof}

\begin{lemma}\label{lem:CLambdaX}
Assume that $(\XL,\sigma_\L)$ is equivalently essentially free
and $\Lambda_\L$ satisfies $\pi$-condition (I).
Then we have 
\begin{equation*}
C(\Lambda_\L)' \cap \RL = C(\XL).
\end{equation*}
\end{lemma}
\begin{proof}
The inclusion relation 
$C(\Lambda_\L)' \cap \RL \supset C(\XL)$ is clear.
The other inclusion relation comes from the preceding lemma under the hypothesis
of equivalently essentially freeness of $(\XL,\sigma_\L)$.  
\end{proof}

\begin{theorem}\label{thm:GRmain3} 
Suppose that two subshifts $\Lambda_i, i=1,2$ 
satisfy  $\pi$-condition (I), and
their canonical $\lambda$-graph bisystems $\L_i =(\L_i^-,\L_i^+),i=1,2$ 
are equivalently essentially free. 
Then the two subshifts  $\Lambda_1, \Lambda_2$ are topologically conjugate
if and only if  there exists an isomorphism  
$\Phi:  \R_{\Lambda_1} \longrightarrow \R_{\Lambda_2}$ 
of $C^*$-algebras such that 
$$
\Phi(C(\Lambda_1)) = C(\Lambda_2), \qquad
\Phi\circ \gamma_{{\Lambda_1}_t} 
= \gamma_{{\Lambda_2}_t}\circ\Phi, \quad t \in \mathbb{T}.
$$
\end{theorem}
\begin{proof}
The only if part comes from Theorem \ref{thm:GRmain2}.
The if part is also from Theorem \ref{thm:GRmain2} 
together with Lemma \ref{lem:CLambdaX}.
\end{proof}

Let us restrict  our subshifts to topological Markov shifts.
Let $(\Lambda_A,\sigma_A)$ be the topological Markov shift 
defined by a square matrix $A$ with entries in $\{0,1\}$.
We denote the $C^*$-algebra $\R_{\mathcal{L}_{\Lambda_A}}$
and its action
$\gamma_{{\Lambda_A}_t}, t \in \mathbb{T}$ 
by
$\R_A$ and $\gamma_{A_t}$
respectively.
For a topological Markov shift  $(\Lambda_A,\sigma_A)$, 
the factor map $\pi: \X_{\Lambda_A} \longrightarrow \Lambda_A$
becomes injective and hence an isomorphism.
The canonical $\lambda$-graph bisystem 
$\L_{\Lambda_A}$ is regarded as the original directed graph defined by the matrix $A$.
Hence if the matrix $A$ is irreducible and not permutation,
then equivalently essentially freeness of $\L_{\Lambda_A}$ and $\pi$-condition (I)
are automatically satisfied. 
We have then the following theorem as a corollary of the above theorem,
that had been already shown in \cite{Ma2019JOT}.
\begin{theorem}[\cite{Ma2019JOT}] \label{thm:mainmarkov}
Let $A, B$ be irreducible, non-permutation  matrices with entries in $\{0,1\}$. 
Then the following two conditions are equivalent:
\begin{enumerate}
\renewcommand{\theenumi}{\roman{enumi}}
\renewcommand{\labelenumi}{\textup{(\theenumi)}}
\item
Their two-sided topological Markov shifts  
$(\Lambda_A, \sigma_A)$ and $(\Lambda_B, \sigma_B)$ are topologically conjugate.
\item
There exists an isomorphism 
$\Phi: \R_A \longrightarrow \R_B$
of $C^*$-algebras such that
$\Phi(C({\Lambda_A})) =C({\Lambda_B})$
and $\Phi\circ\gamma_{A_t}=\gamma_{B_t}\circ\Phi$,
$t \in \T$.
\end{enumerate}
\end{theorem}
Hence the triplet
$
(\R_{A}, C({\Lambda_A}), \gamma_{A})
$
is a complete invariant for topological conjugacy for 
two-sided topological Markov shifts
$(\Lambda_A, \sigma_A)$.
We note that the corresponding result for the Cuntz--Krieger algebras 
are seen in Cuntz-Krieger\cite{CK} and  Carlsen-Rout \cite{CR}.


\medskip

{\it Acknowledgments:}
This work was  supported by JSPS KAKENHI Grant Numbers 15K04896, 19K03537.




\end{document}